\documentclass[12pt]{amsart}

\usepackage[english]{babel}
\usepackage[utf8x]{inputenc}
\usepackage[T1]{fontenc}

\usepackage[margin=1.15in]{geometry}

\usepackage{amsmath}
\usepackage{amstext}
\usepackage{amsfonts}
\usepackage{amssymb}
\usepackage{amsthm}
\usepackage{amsrefs}
\usepackage{mathtools}
\usepackage{dsfont}
\usepackage{color}
\usepackage{graphicx}
\usepackage[colorinlistoftodos]{todonotes}
\usepackage[colorlinks=true, allcolors=blue]{hyperref}
\usepackage{float}

\allowdisplaybreaks

\makeatletter
\let\@wraptoccontribs\wraptoccontribs
\makeatother

\newtheorem{theorem}{Theorem}[section]
\newtheorem{lemma}[theorem]{Lemma}
\newtheorem{prop}[theorem]{Proposition}

\newtheorem{cor}[theorem]{Corollary}
\newtheorem{thm}[theorem]{Theorem}
\newtheorem{lem}[theorem]{Lemma}

\newtheorem*{cor*}{Corollary}
\newtheorem*{thm*}{Theorem}
\newtheorem*{lem*}{Lemma}

\newtheorem*{prop*}{Proposition}

\theoremstyle{definition}
\newtheorem{definition}[theorem]{Definition}
\newtheorem{defn}[theorem]{Definition}
\newtheorem{example}[theorem]{Example}

\newtheorem*{defn*}{Definition}

\newcommand{\define}[1]{{\it #1}}
\theoremstyle{remark}
\newtheorem{remark}[theorem]{Remark}

\newcommand{\pr}{\operatorname{Prob}}

\newcommand{\acts}{\curvearrowright}
\newcommand{\Ad}{\operatorname{Ad}}

\newcommand{\cA}{\mathcal{A}}
\newcommand{\cB}{\mathcal{B}}
\newcommand{\cC}{\mathcal{C}}
\newcommand{\cD}{\mathcal{D}}
\newcommand{\cE}{\mathcal{E}}
\newcommand{\cF}{\mathcal{F}}

\newcommand{\cH}{\mathcal{H}}

\newcommand{\cL}{\mathcal{L}}
\newcommand{\cP}{\mathcal{P}}

\newcommand{\cS}{\mathcal{S}}
\newcommand{\cU}{\mathcal{U}}

\newcommand{\bC}{{\mathbb{C}}}
\newcommand{\bE}{{\mathbb{E}}}
\newcommand{\bF}{{\mathbb{F}}}

\newcommand{\bN}{{\mathbb{N}}}
\newcommand{\bZ}{{\mathbb{Z}}}
\newcommand{\bP}{{\mathbb{P}}}

\newcommand{\sH}{{\bold{H}}}

\newcommand{\sG}{{\bold{G}}}
\newcommand{\sV}{{\bold{V}}}

\newcommand{\VN}{\mathrm{VN}}

\newcommand{\csr}{C^{*}_{\lambda}(\Gamma)}

\newcommand{\malg}{MALG(\nu)}

\newcommand{\nor}{\triangleleft}

\newcommand{\ep}{\varepsilon}
\newcommand{\om}{\omega}
\renewcommand{\ll}{\ell^\infty}
\newcommand{\bnd}{\operatorname{{\bf bnd}}}
\newcommand{\fix}{\mathrm{Fix}}
\newcommand{\supp}{\operatorname{Supp}}

\newcommand{\Rad}{\operatorname{Rad}}
\newcommand{\stab}{\operatorname{Stab}}
\newcommand{\Sub}{\operatorname{Sub}}

\title{Stationary $C^*$-dynamical systems}

\author[Y. Hartman]{Yair Hartman}
\address{Yair Hartman\\Department of Mathematics\\ Ben-Gurion University of the Negev}
\email{hartmany@bgu.ac.il}

\author[M. Kalantar]{Mehrdad Kalantar}
\address{Mehrdad Kalantar\\ University of Houston\\USA}
\email{mkalantar@uh.edu}

\contrib[with an appendix by]{Uri Bader, Yair Hartman, and Mehrdad Kalantar}
\address{Uri Bader\\Weizmann Institute, Rehovot}
\email{uri.bader@gmail.com}

\date{}

\begin{document}

\begin{abstract}
We introduce the notion of stationary actions in the context of $C^*$-algebras. 
We develop the basics of the theory, 
and provide applications to several ergodic theoretical and 
operator algebraic rigidity problems.
\end{abstract}

\thanks{Most of this work was carried out while YH was at Northwestern University. YH was partially supported by the ISF (grant No. 1175/18).}
\thanks{MK was partially supported by the NSF Grant DMS-1700259. }

\maketitle



\section{Introduction}\label{Introduction}
Stationary actions provide a framework that includes all measure preserving 
actions, as well as their opposite systems, boundary actions. 
This framework is general enough to not suffer from an existential 
problem as in the case of invariant measures for non-amenable 
groups, and yet enjoy having enough 
meaningful structural properties.

In the setup of unique stationary dynamical systems, 
random walk theory forms connection between 
topological and measurable dynamics. 
Study of these systems, specially in the noncommutative setting,
is one of the main objectives of this work.
In the case of non-amenable group actions, 
unique stationarity is sometimes more suitable replacement 
for the notion of unique ergodicity, even in the presence of 
invariant measures, as our results in this paper show.

We introduce the notion of stationary $C^*$-dynamical systems, 
in order to develop new tools in the study of operator algebras 
associated to non-amenable groups.
This includes traceless $C^*$ and von Neumann algebras, for which 
many of powerful techniques from the finite-type theories are 
not applicable.

Let $\mu \in \pr(\Gamma)$ be a probability measure on  
a countable discrete group $\Gamma$, $\cA$ be a unital $C^*$-algebra, 
and let $\Gamma\acts\cA$ by $*$-automorphisms.
\begin{defn*}
A state $\tau$ on $\cA$ is 
said to be \define{$\mu$-stationary} if 
$\sum_{g\in \Gamma} \mu(g) g\tau = \tau$. 
\end{defn*}

We are particularly interested in inner actions.
The underlying philosophy here is to view a $C^*$-algebra 
not only as a single structure, but rather 
as a noncommutative dynamical system via the action of its unitary group 
by inner automorphisms. 
Non-triviality of this action is an exclusive feature of noncommutative 
$C^*$-algebras.
From this point of view, a trace is a noncommutative invariant probability measure, 
and admitting a unique trace is the noncommutative counterpart of 
unique ergodicity.
Thus, stationary states are generalizations of traces, 
which in contrast, always do exist. 
We will see that despite this level of generality, they still 
reveal many meaningful structural properties, and in fact,
provide a context in which 
techniques from measurable ergodic theory,
in particular random walks, can be applied to 
study $C^*$-algebras associated to discrete groups. 
This is not entirely in line with the conventional expectation that 
topological dynamics interact with $C^*$-algebra theory, and 
measurable actions with von Neumann algebra theory.
Indeed, we introduce new techniques to use measurable boundaries 
in certain $C^*$-algebraic rigidity problems, and 
we also obtain von Neumann algebraic relative superrigidity results by 
using topological boundaries. 

Another new perspective that this theory provides is an appropriate notion of minimality for noncommutative actions. The current conventional notion of minimality of an action $\Gamma\acts \cA$ of a group $\Gamma$ on a $C^*$-algebra $\cA$ is that $\cA$ does not contain any non-trivial closed proper $\Gamma$-invariant ideals. Although this definition allows generalizations to the noncommutative setting of some results on commutative minimal actions, but it fails to be an appropriate counterpart of minimality in many cases. For instance, with this definition, the trivial action of any group on any simple $C^*$-algebra is minimal, and consequently, a ``factor'' of a minimal action may not be minimal. Alternatively, for an action $\Gamma\acts X$ on a compact space $X$, one observes that the action is minimal if and only if every $\mu$-stationary probability $\nu$ on $X$ has full support, for any generating probability $\mu$ on $\Gamma$. The similar property for an action $\Gamma\acts \cA$ appears to be a more natural notion of minimality. For instance, the trivial action on a non-trivial $C^*$-algebra is never minimal in this sense, and ``factors'' of minimal actions are minimal. In a recent work \cite{AmruKal} which was completed after the first version of this work was published, Amrutam and the second-named author used this perspective and the results of this paper to prove simplicity of all intermediate $C^*$-subalgebras of crossed products of minimal $C^*$-simple actions. 

Topological and measurable boundary actions 
were introduced by Furstenberg in his seminal work 
\cite{Furstenberg-63, Furstenberg-73}
in the context of rigidity of Lie groups.
These notions have recently turned out to be particularly relevant in 
questions of uniqueness of the canonical trace.
The latter is closely related to 
several rigidity problems in ergodic theory and operator algebras 
\cite{Creutz-Peterson-14, Peterson-character-rigidity-15, Kalantar-Kennedy-17, Breuillard-Kalantar-Kennedy-Ozawa-17}.

For instance, the problem of classifying the groups 
with the unique trace property, which had been open for 
almost 40 years, was finally settled 
by Breuillard, Kennedy, Ozawa, and the second named author in
\cite{Breuillard-Kalantar-Kennedy-Ozawa-17}, 
where a characterization of this property was proven in 
terms of existence of faithful topological boundary actions. 
The original proof in \cite{Breuillard-Kalantar-Kennedy-Ozawa-17} 
used the notion of injective envelopes, but 
a simpler proof was provided soon after by 
Haagerup  in \cite[Theorem 3.3]{Haagerup-C*-simple-15}.
In fact, Haagerup's proof has been very inspiring for our work, 
as it clearly shows why boundary actions 
are effective in this type of problems. 
A very closely related problem to the above is the 
classification of $C^*$-simple groups.
A group $\Gamma$ is called \emph{$C^*$-simple} if $\csr$, 
the reduced $C^*$-algebra of $\Gamma$, is simple, meaning that 
it has no non-trivial proper closed ideals.
Similarly, after numerous partial results from many works 
over the span of four decades 
the first characterization of $C^*$-simplicity was proven 
by Kennedy and the second named author 
\cite{Kalantar-Kennedy-17} in terms of existence of 
free topological boundary actions.
Therefore, in particular, the above results combined imply 
$C^*$-simplicity is stronger than the unique trace property. 
Finally, Le Boudec \cite{LeBoudec-17} proved the existence of groups with 
faithful topological boundary actions, but no free such actions, 
hence completely settled the question of whether 
$C^*$-simplicity and the unique trace property are equivalent. 

As an application of our theory, we prove a new characterization 
of $C^*$-simplicity in terms of unique stationarity of the canonical trace. 
\begin{thm*}[Theorem~\ref{new-characterization-C*-simplicity}] 
A countable discrete group $\Gamma$ is $C^*$-simple 
if and only if 
there is $\mu\in\pr(\Gamma)$
such that the canonical trace
is the unique $\mu$-stationary state on $\csr$.
\end{thm*}
In particular, this result shows that $C^*$-simplicity is also a uniqueness 
property of the canonical trace.
This is indeed quite natural with our point of view that
$\csr$ is rather a $\Gamma$-$C^*$-algebra via the inner action:
every ideal is invariant, and 
therefore simplicity is a noncommutative minimality problem. 
Now considering the commutative picture, 
since stationary measures always exist, 
existence of a unique stationary probability with full support 
implies minimality.

Also, our above characterization of $C^*$-simplicity 
provides an intrinsic dynamical explanation
for the difference between the unique trace property and $C^*$-simplicity: while
the former corresponds to unique ergodicity, the latter corresponds to unique
stationarity. 
We may even give a manifestation of this in the commutative 
setting: every group $\Gamma$ 
admits an action on a compact metric space such that the 
difference between unique ergodicity and unique stationarity of the action 
translates into the difference between the unique trace property and 
$C^*$-simplicity, as follows.

Let $\Sub_a(\Gamma)$ denote the set of all amenable subgroups 
of $\Gamma$, which is a compact space on which $\Gamma$ 
acts by conjugations. 
Bader, Duchesne and Lecureux~\cite{Bader-Duchesne-Lecureux-16} proved that 
$\Gamma$ has the unique trace property if and only if $\Gamma\acts \Sub_a(\Gamma)$ 
is uniquely ergodic. Here we prove:

\begin{thm*}[Corollary~\ref{cor:sub_a-unq-stationary}]
A countable discrete group $\Gamma$ is $C^*$-simple if and only if there is 
$\mu\in\pr(\Gamma)$ such that the action 
$\Gamma\acts \Sub_a(\Gamma)$ 
is uniquely $\mu$-stationary.
\end{thm*}

We would like to highlight the interesting, and somehow curious fact 
that $C^*$-simplicity, a purely $C^*$-algebraic property, 
would single out certain random walks (or measures) on $\Gamma$ that reveal
its $C^*$-simplicity.
Moreover, it turns out that these measures posses significant 
ergodic theoretical properties, connecting $C^*$-simplicity
to random walks on groups.
Thus, it suggests to consider 
$C^*$-simplicity of $\Gamma$ as rather a property of the measure(s) 
$\mu\in \pr(\Gamma)$ in the above theorems. 
\begin{defn*}
We say that $\mu\in\pr(\Gamma)$
is a \emph{$C^*$-simple measure} if the canonical trace $\tau_0$ is the unique $\mu$-stationary state on $\csr$.
\end{defn*}

For instance, we prove:

\begin{thm*}[Theorem~\ref{Zimmer-amenable-action-C*-Simple-group-is-free}]
Suppose $\mu\in\pr(\Gamma)$ is $C^*$-simple. 
Then any measurable $\mu$-stationary action 
with almost surely amenable stabilizers, is essentially free.

In particular, the action of $\Gamma$ on the Poisson boundary of $\mu$ is essentially free.
\end{thm*}

Hence, we are naturally led to the problem of finding $C^*$-simple measures.
Our proof of the existence of $C^*$-simple measures on $C^*$-simple groups 
does not reveal concrete measures.  
In fact, part of our construction of the $C^*$-simple measure 
follows a similar construction as in Kaimanovich--Vershik's proof of 
Furstenberg's conjecture 
on the existence of measures on amenable groups with trivial Poisson boundary.
The following result, on the other hand, allows verifying 
$C^*$-simplicity of many concrete measures.

\begin{thm*}[Theorem~\ref{free-USB->C*-simple}]
Let $\mu \in \pr(\Gamma)$ and suppose $\Gamma$ admits an essentially free $\mu$-boundary which 
has a compact model that is uniquely $\mu$-stationary.
Then $\mu$ is a $C^*$-simple measure.
\end{thm*}

This result also highlights another advantage of our approach in 
using measurable boundaries in the above problems. 
In contrast to the topological case, measurable boundaries have been 
studied extensively, and there are several powerful methods 
due to the fundamental work of Kaimanovich~\cite{Kaimanovich-00}, 
for realizing these boundaries.
These methods have been resulted 
in many deep realization results 
(e.g. \cite{Kaimanovich-00, Kaimanovich-Masur-96, Brofferio-Schapira-11, Ledrappier-85, Horbez-16}).
In fact, in many examples, a concrete unique stationary model 
of a boundary is provided, 
and under some regularity assumptions on the measure 
the corresponding stationary measure is the Poisson measure.

To further highlight the contrast to the topological case, 
we remark that in the case of non-amenable discrete groups, the Furstenberg boundary is always a 
non-metrizable extremally disconnected space, not concretely identifiable 
in any known example.
We take advantage of the concreteness of the 
topological models of measurable boundaries
in order to verify their essential freeness. 
For example, we prove a \emph{0-1 Law} (Theorem~\ref{zer-one-law}) for a class of stationary actions that 
include algebraic actions, which provides
an essential freeness/triviality dichotomy for such actions. Using that
we prove:

\begin{thm*}[Theorem~\ref{thm:linear}]
Let $\Gamma$ be a finitely generated linear group with 
trivial amenable radical. Then every generating
measure on $\Gamma$ is $C^*$-simple.
\end{thm*}

We obtain this result 
by proving the existence of essentially free mean-proximal actions for linear 
groups. 
A crucial step in the proof is an extension result 
for mean-proximal actions that we prove jointly with Uri Bader (see Appendix~\ref{sec:appendix}). \\

Furthermore, we conclude $C^*$-simplicity of the measures in the 
following contexts.

\begin{thm*}[Example~\ref{ex:hyperbolic}, Theorems~\ref{mapping-class-groups}
and \ref{out(Fn)}]
Every generating measure on a mapping class group or 
a hyperbolic group $\Gamma$ with trivial amenable radical 
is $C^*$-simple. 
The same conclusion holds for finitely supported measures on $\rm{Out}(\bF_n)$.
\end{thm*}

In the last section, we study unique stationarity and unique trace property 
relative to subgroups, and prove several ergodic theoretical relative rigidity results.

\begin{thm*}[Theorem~\ref{charmonic-rigidity-2}]
Let $\mu \in \pr(\Gamma)$ and suppose $\Gamma$ admits an essentially free $\mu$-boundary which 
has a compact model that is uniquely $\mu$-stationary.
Then a $\mu$-stationary action 
is essentially free 
if its restriction to some co-amenable subgroup $\Lambda\leq \Gamma$ is essentially free. 
\end{thm*}

All our results mentioned up to this point, are obtained by applying 
the techniques developed here to use 
measurable boundaries to deduce $C^*$-algebraic rigidity properties.
In contrast, a von Neumann algebraic relative superrigidity 
result below is proven in the last section 
by using topological boundary actions.

\begin{thm*}[Theorems~\ref{relative-op-alg-superrigid} and~\ref{thm:co-amenalbe-are-character-rigid}]
Let $\Gamma$ be a countable discrete group 
that admits a faithful topological boundary,
and let $\Lambda\leq \Gamma$ be an icc co-amenable subgroup. 
Suppose $\pi:\Gamma\to\cU(\cH)$ is a unitary representation such that 
$\pi(\Gamma)''$ is a finite von Neumann algebra. 
If the restriction $\pi|_\Lambda$ extends to a von Neumann algebra isomorphism 
$\pi(\Lambda)'' \cong \cL(\Lambda)$, then $\pi$ extends to a 
von Neumann algebra isomorphism $\pi(\Gamma)'' \cong \cL(\Gamma)$.
\end{thm*}

Recall that by Furman~\cite{Furman-03} 
a group $\Gamma$ admits a faithful 
topological boundary
if and only if it has a trivial amenable radical 
(that is, it has no non-trivial amenable normal subgroups).
As a corollary we obtain the following 
relative version of the result of Stuck-Zimmer.

\begin{thm*}[Theorems~\ref{pmp-free-if-rest-free} and \ref{thm:co-amenalbe-are-character-rigid}]
Let $\Gamma$ be a countable discrete group 
with trivial amenable radical, and let $\Lambda\leq \Gamma$ be a co-amenable subgroup. Then a probability measure preserving action $\Gamma\acts(X, m)$ 
is essentially free if its restriction $\Lambda\acts(X, m)$ is essentially free.
\end{thm*}

In terms of \emph{Invariant Random Subgroup} (IRS), the above is equivalent
to that every IRS of $\Gamma$
intersects $\Lambda$ non-trivially with positive probability. 
In particular, every non-trivial normal subgroup $\mathsf{N}\triangleleft \Gamma$
intersects $\Lambda$ non-trivially.\\

These above results are relative versions of the well studied  operator algebraic superrigidity which originated by Connes' conjecture \cite{Jones-00}. 
The first major result in this direction was obtained by 
Bekka \cite{Bekka-07} and then was generalized by Peterson 
in~\cite{Peterson-character-rigidity-15}
(see also \cite{Peterson-Thom-16, Creutz-Peterson-14}).

Very recently Boutonnet and Houdayer~\cite{BoutHoud} have strengthen Peterson's 
work based on the concept of stationary states developed
here, in this paper.  
This breakthrough result is another manifestation of the importance of stationarity in the noncommutative setup. Indeed, it provides the 
strongest known result in this context (an SRS rigidity
for higher rank lattices) and answers a question of Glasner and Weiss from \cite{Glasner-Weiss-14}.\\

In addition to this introduction, this paper has six other sections. 
In Section~2 we briefly review the requisite background material. 
In Section~3 we recall the definitions of measurable and topological 
boundaries, and prove that a unique stationary measurable boundary is 
a topological boundary.
In Section 4 we introduce stationary $C^*$-dynamical systems, 
prove basic properties, and provide a number of examples. In particular, 
we show how unique stationarity implies $C^*$-simplicity. 
In Section 5 we prove our new characterization of $C^*$-simplicity in 
terms of unique stationarity of the canonical trace. We then prove 
various properties of $C^*$-simple measures, and obtain another characterization
of $C^*$-simplicity in terms of unique stationarity of $\Sub_a$.
Section 6 is concerned with the question of freeness of 
unique stationary actions, and verifying that certain measures are $C^*$-simple.
In Section 7, we apply our techniques to prove several superrigidity results
relative to co-amenable subgroups.\\
The paper also contains an appendix, 
which includes our joint result with Uri Bader, an extension theorem for mean-proximal 
actions that we need in the proof of $C^*$-simplicity of generating measures on finitely generated linear groups.\\

\noindent
\textbf{Acknowledgements.}
We are grateful to Nir Avni and Uri Bader for many fruitful discussions and helpful suggestions during the course of completion of this project. We would also like to thank Pawe\l{} Kasprzak, Matthew Kennedy, Hanfeng Li, Howard Masur, Dan Ursu, and Phillip Wesolek for helpful comments.
We thank the anonymous referee for the careful reading of our paper, for suggesting a more general statement for Theorem~\ref{new-characterization-C*-simplicity-general} compared to the previous version, and for their other helpful comments and suggestions.

\section{Preliminaries}\label{Preliminaries}

Throughout the paper $\Gamma$ is a countable discrete group, 
and $\Gamma\acts X$ denotes an action of $\Gamma$ by homeomorphisms 
on a compact (Hausdorff) space $X$. 
The action $\Gamma\acts X$ is \define{minimal} if $X$ 
has no non-empty proper closed $\Gamma$-invariant subset.
We denote by $\pr(X)$ the set of all Borel probability measures on $X$.
For $\nu \in \pr(X)$ we denote by $\cP_\nu$ its corresponding \define{Poisson map}, i.e.
the unital positive $\Gamma$-equivariant map $\cP_\nu : C(X) \to \ell^\infty(\Gamma)$ defined by
\begin{equation}\label{def:Poisson-map}
\cP_\nu(f)(g) = \int_{X} f(gx) \, d\nu(x), \quad g \in \Gamma,\ f \in C(X).
\end{equation}

We also consider measurable actions, i.e. actions 
$\Gamma\acts(Y, \eta)$ of $\Gamma$ on probability spaces $(Y, \eta)$
by measurable automorphisms. 
A measurable action $\Gamma\acts(Y, \eta)$ is a \emph{non-singular action} 
(or $\eta$ is a non-singular measure) if 
$g\eta$ and $\eta$ are in the same measure class for every $g\in \Gamma$.  
The Poisson map associated to a non-singular measure is defined 
similarly to~\eqref{def:Poisson-map}. 
Throughout the paper, unless otherwise stated, all measurable actions are 
assumed to be non-singular.

A compact $\Gamma$-space $X$ is said to be a compact 
model of a measurable $\Gamma$-space $(Y,\eta)$ if there 
exists a Borel measure $\nu\in \pr(X)$ 
such that $(Y, \eta)$ and $(X,\nu)$ are measurably 
isomorphic as $\Gamma$-spaces. It is a well-known 
fact that every measurable action on a standard Borel space
has a compact model which is metrizable.\\

For a Hilbert space $\cH$ we denote by $\cB(\cH)$ the set of all 
bounded operators on $\cH$. 
A subalgebra $\cA \le \cB(\cH)$ 
is a {$C^{*}$-algebra} if it 
is closed in the operator norm and under taking adjoint. 
In this case, $\cA$ is \define{unital} if it contains the identity operator on $\cH$.

If $X$ is a compact space, then $C(X)$ with the $\sup$-norm and the complex conjugate as the involution is a $C^*$-algebra, and 
conversely, by Gelfand's representation theorem, any unital commutative $C^*$-algebra is of this form. 
Hence unital $C^*$-algebras are viewed as algebras of continuous functions on 
``noncommutative compact spaces''.

An element $a$ in a $C^*$-algebra $\cA$ is said to be \emph{positive}, written $a\geq 0$, 
if $a=b^*b$ for some $b\in \cA$. We denote by $1_{\cA}$ the unit element in $\cA$.

A linear map $\phi:\cA\to\cB$ between $C^*$-algebras is 
\define{positive} if it sends positive elements to positive elements, 
and it is \define{unital} if $\phi(1_\cA)=1_\cB$.
A linear $\phi$ is a \emph{$*$-homomorphism} if it is  
multiplicative and $\phi(a^*) = \phi(a)^*$ for all $a\in\cA$, 
and it is a \emph{$*$-isomorphism} if it is moreover bijective.
We denote by $\mathrm{Aut}(\cA)$ the group of all $*$-automorphisms on $\cA$.

The noncommutative counterpart of probability measures 
are \emph{states} on $C^*$-algebras.
A {state} on $\cA$ is a positive linear functional $\rho:\cA \to \bC$ 
with $\rho(1_\cA)=1$. 
We denote by $\cS (\cA)$ the (convex, weak*-compact) 
space of all states on $\cA$.
A state $\rho$ is \define{faithful} if
$\rho(a)>0$ for any non-zero positive element $a$. 
A state $\tau\in\cS(\cA)$ is a \define{trace} if $\tau(ab) = \tau(ba)$
for all $a, b \in \cA$.
Obviously every state on a commutative $C^*$-algebra is a trace,
but on the other hand there are $C^*$-algebras that do not
admit any trace.

Let us recall the GNS construction associated to a state 
$\rho\in\cS(\cA)$. Define the sesquilinear form 
$\langle a, b \rangle_\rho := \rho (b^*a)$ on $\cA$, let 
$L^2(\cA, \rho)$ be the induced Hilbert space, and for every $b\in\cA$ denote by $\widehat b$ its corresponding element in $L^2(\cA, \rho)$.
Then the GNS representation $\pi_\rho :\cA\to\cB(L^2(\cA, \rho))$ 
is defined by $\pi_\rho(a) (\widehat b) = \widehat{ab}$, for $a, b \in \cA$.

A $C^*$-algebra is said to be \emph{simple} if it does not contain 
any non-trivial proper closed two-sided ideal.\\

A {von Neumann algebra} is a $C^*$-algebra $M$ 
that is also a dual Banach space. In this case the predual $M_*$ is
unique. A bounded linear functional on $M$ is called 
\define{normal} if it belongs to $M_*$.
Since $\cB(\cH)$ itself is a dual Banach space (the predual
being the space of trace-class operators), it follows that 
a unital $C^*$-subalgebra $\cA\subset \cB(\cH)$ is a von Neumann
algebra if and only if it is closed in the weak* topology of $\cB(\cH)$,
or equivalently closed in the weak or strong operator topologies.
By von Neumann's bicommutant theorem 
a self-adjoint unital subalgebra $M\subset \cB(\cH)$ is a von Neumann
algebra if and only if $M'' = M$, where $M'=\{x\in \cB(\cH) : xy=yx \text{ for all } y\in M\}$
is the commutant of $M$ in $\cB(\cH)$, and $M'' = (M')'$.

If $(X, \nu)$ is a probability space, then $L^\infty(X, \nu)$
is a von Neumann algebra, and every commutative von Neumann algebra
is of this form. Hence,
von Neumann algebras are viewed as algebras
of essentially bounded measurable functions on ``noncommutative
probability spaces''.

The GNS representation associated to normal states on a 
von Neumann algebra is defined similarly as in the $C^*$-algebra 
case.

\subsection{Random walks and stationary dynamical systems} \label{Random walks and stationary dynamical systems}
The theory of random walks on groups and their associated boundaries
was introduced by Furstenberg~\cite{Furstenberg-63, Furstenberg-73}, and later studied extensively by various people. This theory 
provides a framework to apply probabilistic ideas and methods in
the study of analytic properties of groups.
Let us briefly recall the notion of random walks on discrete groups. We
refer the reader to~\cite{Furstenberg-73, Furman-book-02, Bader-Shalom-06} for more details.
 
Let $\mu\in \pr(\Gamma)$ be \define{generating}, i.e.
$\Gamma$ is the semigroup generated by $\supp(\mu)=\{g : \mu(g)>0\}$.
The \define{random walk on $\Gamma$ with law $\mu$} 
(or just the $(\Gamma, \mu)$-random walk) is 
the time-independent Markov chain with state space $\Gamma$, 
initial distribution $\delta_e$ (the Dirac probability measure supported at 
the neutral element $e\in \Gamma$), 
and transition probabilities $p(g, h) = \mu(g^{-1} h)$ for $g, h\in \Gamma$.
Thus, the probability of walking on the path $e, g_1, g_1g_2, \dots, g_1g_2\cdots g_k$
on the first $k+1$ steps is $\mu(g_1)\mu(g_2)\cdots\mu(g_k)$.

The \define{space of paths} of the random walk is the probability space
$(\Omega, \mathbb{P}_\mu)$, where
$\Omega =\Gamma^\mathbb{N}$ and
$\mathbb{P}_\mu$ is the Markovian measure, that is, the unique probability on $\Omega$ defined by
\[
\mathbb{P}_\mu\left(\{\om\in\Omega : \om_1 = g_1, \om_2=\om_1 g_2, \dots, \om_k= \om_{k-1} g_k\}\right)
= \mu(g_1)\mu(g_2)\cdots\mu(g_k)
\]
for $g_1, g_2, \dots, g_k \in \Gamma$ and $k\in \bN$. 

For a fixed $g\in \Gamma$ it follows that
$\bP_\mu(\{\om\in\Omega : \om_n = g\})$, the probability of the random
walk being in position $g$ at the $n$-th step, is equal to $\mu^n(g)$,
where $\mu^n$ is the $n$-th convolution power of $\mu$.

Alternatively, $\mathbb{P}_\mu$ can be described as the push-forward of the Bernoulli measure 
$\mu \times \mu \times \cdots$
under the transformation 
$\Gamma^\mathbb{N} \to \Gamma^{\mathbb{N}\cup\{0\}}$, $(g_k) \mapsto (\om_k)$,
where $\om_0 = e$ and $\om_k = g_1 g_2\cdots g_k$
for $k\in\bN$. In probabilistic terms, $g_k$ is the increment
and $\omega_k$ is
the position of the random walk at time $k$.

Suppose $\Gamma\acts X$ is a continuous action on a compact space,
and let $\mu\in\pr(\Gamma)$.
A Borel probability measure $\nu$ on $X$ is called \define{$\mu$-stationary} if 
\[
\nu = \sum_{g\in \Gamma} \mu(g)\, g\nu .
\]
In this case we say $(X, \nu)$ is a $(\Gamma, \mu)$-space, and we write $(\Gamma, \mu)\acts (X, \nu)$.

The basic feature of stationary measures is their existence.
Unlike invariant measures, on any compact $\Gamma$-space $X$ there exists at least one $\mu$-stationary measure. 
While the existence holds for arbitrary compact space, significant
part of the theory is developed in the context of metrizable
compact spaces. The following is the fundamental result 
in this context that builds the connection between stationary systems and the theory of random walks. 

\begin{thm}[Furstenberg]\label{cond-meas}
Let $\nu$ be a $\mu$-stationary measure on a metrizable compact $\Gamma$-space $X$. 
Then for $\bP_\mu$-almost every path
$\omega\in\Omega$ the limit $\nu_\omega:=\mathrm{weak^*-}\lim_n \omega_n \nu$ exists. 
Moreover, we have 
\[
\nu = \int_\Omega \nu_\omega\, d\bP_\mu (\omega) .
\]
\end{thm}
The measures $\nu_\omega$ are called the conditional measures. 

We will recall and prove other basic facts about stationary dynamical systems
in the more general setting of actions on $C^*$-algebras in later sections.\\

Stationary actions are also defined in measurable setting. 
A non-singular action $\Gamma \acts (Y, \eta)$ is $\mu$-stationary if 
$\sum_{g\in \Gamma} \mu(g) \, g\eta = \eta$.
Note that if $X$ is a compact $\Gamma$-space, and 
$\mu\in\pr(\Gamma)$ is generating, then any $\mu$-stationary 
$\nu\in \pr(X)$ is non-singular.

\subsection{$C^*$-dynamical systems}
We briefly recall the notion of group actions on $C^*$-algebras 
and establish the notation and terminology that we will be using 
in the sequel. We refer the reader to \cite{Brown-Ozawa-08} 
for more details.

A unital $C^{*}$-algebra $\cA$ is called a \define{$\Gamma$-$C^{*}$-algebra} 
if there is an action $\alpha:\Gamma \acts \cA$ 
of $\Gamma$ on $\cA$ by $*$-automorphisms,
that is, $\alpha$ is a group homomorphism $\Gamma \to \mathrm{Aut}(\cA)$.
A class of examples of $\Gamma$-$C^{*}$-algebras that are
of main interest to our work in this paper are obtained as follows.
Let $\pi:\Gamma \to \cU(\cH_\pi)$ be a unitary representation on the Hilbert space
$\cH_\pi$ (where $\cU(\cH_\pi)$ is the group of unitary operators on $\cH_\pi$). Then $\Gamma$ acts on $\cB(\cH_\pi)$ by inner automorphism
$\Ad_g(x) := \pi(g)x\pi(g^{-1})$, $g\in \Gamma$, $x\in \cB(\cH_\pi)$,
as well as on any $C^*$-algebra $\cA\leq \cB(\cH_\pi)$ that is invariant under
this action. In fact, every $\Gamma$-$C^{*}$-algebra is formed
in the above fashion for some unitary representation of $\Gamma$.

In particular, for any unitary representation $\pi:\Gamma \to \cU(\cH_\pi)$, 
the group $\Gamma$ acts on the $C^*$-algebra 
$C^{*}_{\pi}(\Gamma) := \overline{\operatorname{span}\{\pi(g) : g\in\Gamma\}}^{\|\cdot\|}\subset \cB (\cH_\pi)$ 
by inner automorphisms. Throughout the paper 
$\Gamma\acts C^{*}_{\pi}(\Gamma)$ denotes this action 
unless otherwise stated. 

An important example is the left regular representation $\lambda : \Gamma \to \cU(\ell^2(\Gamma))$
defined by $(\lambda_g\xi)(h) = \xi(g^{-1}h)$, $h\in\Gamma$ and 
$\xi\in \ell^2(\Gamma)$. In this case $C^{*}_{\lambda}(\Gamma)$ 
is called the \define{reduced $C^*$-algebra of $\Gamma$}.

The \define{full $C^*$-algebra} $C^*(\Gamma)$ 
of $\Gamma$ is the universal $C^*$-algebra generated 
by $\Gamma$ in the sense that for any unitary 
representation $\pi:\Gamma \to \cU(\cH_\pi)$ 
there is a canonical surjective $*$-homomorphism 
$C^*(\Gamma)\to C^{*}_{\pi}(\Gamma)$. 
We consider $\Gamma$ as a subset of $C^*(\Gamma)$ 
in the natural way.

Similarly to compact spaces and probability measures (``the commutative case''), any action $\Gamma\acts \cA$ 
induces an adjoint action $\Gamma \acts \cS (\cA)$. 
We denote by $\cS_\Gamma(\cA)$ the simplex of all $\Gamma$-invariant states,
that is, states $\nu$ such that $g\nu=\nu$ for all $g\in\Gamma$. It is 
obvious that $\cS_\Gamma(\cA)$ is compact in the weak* topology. 

In the case of $\Gamma\acts C^*_\pi(\Gamma)$ by inner automorphisms, 
$\cS_\Gamma(C^*_\pi(\Gamma))$ coincides with the set of all traces on $C^*_\pi(\Gamma)$. 
In particular, in the case of the reduced $C^*$-algebra, $\cS_\Gamma(\csr)$
is never empty as it contains \define{the canonical trace} $\tau_0$ 
(or $\tau_0^\Gamma$ if we need to clarify its association to $\Gamma$), 
namely the extension of the linear functional $\sum c_g \lambda_g \mapsto c_e$.

Similarly to the $C^*$-algebra case, a source of
examples of $\Gamma$-von Neumann algebras for us
is by means of unitary representation theory of groups.
Let $\pi:\Gamma \to \cU(\cH_\pi)$ be a unitary representation on the Hilbert space
$\cH_\pi$. Then we have
$\Gamma\acts VN_{\pi}(\Gamma)$ by inner automorphism, where 
$VN_{\pi}(\Gamma):= \{\pi(g) : g\in\Gamma\}'' = \overline{\operatorname{span}\{\pi(g) : g\in\Gamma\}}^{\text{weak*}}\subset \cB (\cH_\pi)$, is {the von Neumann algebra generated by 
the representation $\pi$}.
In the case of the left regular representation, 
$\VN_{\lambda}(\Gamma) = \overline{\csr}^{\text{weak*}}$ 
is called the \define{group von Neumann algebra} of $\Gamma$
and is denoted by $\cL(\Gamma)$.
The canonical trace $\tau_0$ extends to a normal trace on $\cL(\Gamma)$.

\subsection{Crossed product $C^*$-algebras}
The bridge between the theories of operator algebras and
(topological or measure-theoretical) dynamics is made by the \emph{crossed product} construction. 
Loosely speaking, the crossed product $C^*$-algebra associated to an
action $\Gamma\acts\cA$ is a $C^*$-algebra that contains,
and is generated by, copies of $\cA$ and $\Gamma$ such that
the action $\Gamma\acts\cA$ in this bigger algebra is by inner automorphisms.

We recall the more precise definition below and refer the 
reader to~\cite{Brown-Ozawa-08} for more details.

Let $\Gamma\acts\cA$, and
consider the Hilbert space
$\ell^2(\Gamma, \cH) = \{\xi: \Gamma\to \cH \,|\, \sum_{g\in\Gamma} \|\xi(g)\|^2<\infty\}$. 
For $g\in \Gamma$ and $a\in \cA$ define the operators 
$\tilde\lambda(g), \iota(a) \in \cB(\ell^2(\Gamma, \cH))$
by $(\tilde\lambda(g)\xi) (h) = \xi(g^{-1}h)$ 
and $(\iota(a)\xi) (h) = (h^{-1}a)\xi(h)$. Then 
the $C^*$-subalgebra of $\cB(\ell^2(\Gamma, \cH))$ generated by
$\tilde\lambda(\Gamma)$ and $\iota(\cA)$ is called
the \define{reduced crossed product}
of the action $\Gamma\acts \cA$,
and is denoted by $\Gamma\ltimes_r \cA$.
It can also be seen that $C^*(\{\tilde\lambda(g) : g\in\Gamma\})$
is canonically isomorphic to $\csr$.

The following simple lemma, which generalizes~\cite[Lemma 2]{Archbold-Spielberg-94}, 
and in fact follows from its proof, 
is key in allowing passage between 
classical and noncommutative settings.
We present a more elementary proof which was provided to us by Hanfeng Li. We thank him for this, as well as for pointing out to us 
the crucial fact that 
\cite[Lemma 2]{Archbold-Spielberg-94} 
is also valid for actions on operator systems.

\begin{lemma}\label{lemma0}
Let $\pi:\Gamma\to \cU(\cH_\pi)$ be a unitary representation.
Suppose $\rho$ is a state on $\cB(\cH_\pi)$ and $g\in \Gamma$.
If there is $a\in \cB(\cH_\pi)$ with $0\leq a \leq 1$ such that 
$\rho(a)= 1$ and $\rho(\pi(g^{-1})a\pi(g)) = 0$, then $\rho(\pi(g)) = 0$.
\end{lemma}

\begin{proof}
Since 
$0\leq \pi(g^{-1})a\pi(g)\leq1$ and $\rho(\pi(g^{-1})a\pi(g)) = 0$ we get
$\rho(\pi(g^{-1})a^2\pi(g)) = 0$.
Similarly, 
$\rho((1-a)^2) = 0$,
and the Cauchy-Schwartz inequality implies 
$|\rho((1-a)\pi(g))| \leq \rho((1-a)^2) \rho(1) = 0$.
So we have
\begin{align*}
|\rho(a\pi(g)) |&= |\rho(\pi(g) \pi(g^{-1})a\pi(g)) |\\&= |\rho(\pi(g^{-1})a\pi(g)\pi(g^{-1}))| \\
&\leq \rho(\pi(g^{-1})a^2\pi(g)) = 0 .
\end{align*}
Hence $\rho(\pi(g)) = \rho(a\pi(g)) +  \rho((1-a)\pi(g)) = 0$.
\end{proof}

\subsection{Positive definite functions, invariant and stationary random subgroups}
A function $\phi:\Gamma\to \bC$ is called a \define{positive definite function (pdf)} 
if for any $n\in \bN$ and $g_1, g_2, \dots, g_n \in \Gamma$ the matrix 
$[\phi(g_ig_j^{-1})]_{i,j = 1, \dots, n}$ is positive.

If $\pi$ is a unitary representation of $\Gamma$ on a Hilbert space $\cH_\pi$, 
then for any state $\rho$ on $C^*_\pi(\Gamma)$ the function 
$\phi(g) = \rho(\pi(g))$ is a pdf on $\Gamma$.
Conversely, if $\phi$ is a pdf on $\Gamma$ then 
there is a unitary representation $\pi$  
(e.g. the GNS representation associated to $\phi$ \cite{Brown-Ozawa-08}),
and a vector $\xi\in\cH_\pi$ such that 
$\phi(g) = \langle \pi(g)\xi, \xi\rangle_{\cH_\pi}$.
This yields a canonical identification between the 
weak* compact convex space $\cS(C^*(\Gamma))$ of all states on the full $C^*$-algebra of $\Gamma$, and the 
space $P_\Gamma$ of all pdf $\phi:\Gamma\to\bC$ normalized by $\phi(e)=1$, 
endowed with the pointwise convergence topology. 
In this correspondence, $\rho$ is a {trace} if and only if $\phi$ is 
a \emph{character}, i.e.\ a normalized conjugation invariant pdf 
(i.e.\ $\phi(h^{-1}gh) = \phi(g)$ for all $g, h\in \Gamma$).\\

Let $\Sub(\Gamma)$ be the space of all subgroups of $\Gamma$ with the topology inherited from $2^\Gamma$ (known also as the \emph{Chabauty topology}). It is a compact space, on which $\Gamma$ 
acts by conjugation $g.\Lambda = g^{-1}\Lambda g$.
A $\Gamma$-invariant Borel probability measure $\eta$ is 
called an \define{invariant random subgroup (IRS)~\cite{Abert-Glasner-Virag-14, 7Samurai}}.
If $\eta$ is only $\mu$-stationary for some $\mu\in\pr(\Gamma)$, 
we say that $\eta$ is a \define{$\mu$-stationary random subgroup ($\mu$-SRS)}.

We denote by $\Sub_a(\Gamma)$ the closed, $\Gamma$-invariant
subset of $\Sub(\Gamma)$ of all {amenable} subgroups.

\begin{lem}\label{random-subgrp-->pos-def}
Let $\eta$ be Borel probability measure on $\Sub(\Gamma)$. Then the function $\phi_{\eta}(g) = \eta(\{\Lambda : g\in \Lambda\})$ is positive definite. 
If moreover, $\eta$ is supported on $\Sub_a(\Gamma)$, then there is a state $\rho$ 
on the reduced $C^*$-algebra 
$\csr$ such that $\phi_{\eta}(g) = \rho(\lambda_g)$ for every $g\in \Gamma$.
\end{lem}

\begin{proof}
Given a subgroup $\Lambda\in \Sub(\Gamma)$, let $\mathds{1}_{\Lambda}$
denote its characteristic function. It is not hard to see that
$\mathds{1}_{\Lambda} \in P_\Gamma$. The map 
$\Sub(\Gamma)\ni\Lambda \mapsto \mathds{1}_{\Lambda}\in P_\Gamma$
is clearly continuous. Let
$\bar{\eta}\in \pr(P_\Gamma)$ be the push-forward of $\eta$, 
then $\phi_{\eta}$ is the barycenter of $\bar{\eta}$.

If $\Lambda\leq \Gamma$ is an amenable subgroup, then 
the quasi-regular representation $\Gamma$ on 
$\ell^2(\Gamma\backslash\Lambda)$ is weakly contained in 
the regular representation of $\Gamma$, which implies 
$\mathds{1}_{\Lambda}$ corresponds to a state on $\csr$. Hence if $\eta$
is supported on amenable subgroups then the barycenter
of $\bar{\eta}$ corresponds to a state $\rho$ on $\csr$.
\end{proof}

\section{Topological, measurable, and uniquely stationary boundaries}\label{Topological vs measuableboundaries}

In this section we recall the notions of topological and measurable boundary actions of discrete groups. 
We comment on advantages of each setting over the other, and 
prove that a uniquely stationary measurable boundary (USB) is  
also a topological boundary. Thus, in the framework of 
such systems we may apply both topological and 
measure-theoretical techniques.

\subsection{Topological vs. measurable boundaries}\label{sec:top-and-measurable-bnds}

For more details on theory of boundary actions we 
refer the reader to~\cite{Furstenberg-73, Glasner-76, Furman-book-02}.

\begin{defn}[Topological boundary actions]
A continuous action $\Gamma\acts X$ 
on a compact space $X$ is 
a \define{topological boundary action} if 
for every $\nu\in \pr(X)$ and $x\in X$
there is a net $g_i$ of elements of $\Gamma$ 
such that $g_i \nu \to \delta_x$ in the weak* topology,
where $\delta_x$ is the Dirac measure at $x$.
\end{defn}

It can be shown that an action $\Gamma\acts X$ 
is a topological boundary if and only if 
for every $\eta\in\pr(X)$ the Poisson map 
$\cP_\eta$ is isometric (\cite{Azencott-70}).

\begin{prop}\cite{Furstenberg-73}
There is a unique (up to $\Gamma$-equivariant homeomorphism) 
maximal $\Gamma$-boundary $\partial_F \Gamma$
in the sense that every $\Gamma$-boundary $X$ is a continuous 
$\Gamma$-equivariant image of $\partial_F \Gamma$.
\end{prop}

The maximal boundary $\partial_F \Gamma$ is called the \define{Furstenberg boundary} of $\Gamma$.

\begin{defn}[Measurable boundary actions]\label{defn:measurable-boundaries}
Let $\mu\in\pr(\Gamma)$, and suppose $\nu$ is 
a $\mu$-stationary measure on a metrizable $\Gamma$-space $X$. The action 
$(\Gamma, \mu)\acts (X,\nu)$
is a \define{$\mu$-boundary action} if
for almost every path 
$\omega = (\omega_k)\in \Omega$ of the $(\Gamma, \mu)$-random walk,
the sequence $\omega_k\nu$ converges to a Dirac measure $\delta_{x_\om}$.

In this case the map $\bnd : (\Omega, \bP_\mu) \to (X, \nu)$ 
defined by $\bnd(\om) = x_\om$ is called a \emph{boundary map}.

A measurable non-singular action $\Gamma\acts (Y, \eta)$ is called a 
$\mu$-boundary action if it is $\mu$-stationary and admits a compact metrizable model 
$(X,\nu)$ which is a $\mu$-boundary in the above sense.
\end{defn}

\begin{prop}\cite{Furstenberg-73}
There is a unique (up to $\Gamma$-equivariant measurable isomorphism) 
maximal $\mu$-boundary $(\Pi_\mu , \nu_\infty)$
in the sense that every $\mu$-boundary $(X,\nu)$ is a measurable 
$\Gamma$-equivariant image of $(\Pi_\mu , \nu_\infty)$.
\end{prop}

The maximal $\mu$-boundary $(\Pi_\mu , \nu_\infty)$ 
is called the \define{Poisson boundary} of the pair $(\Gamma,\mu)$ (also known sometimes as the \emph{Furstenberg-Poisson boundary}).

One should note that since the Poisson boundary is 
defined up to measurable isomorphism, it should be considered 
as a measurable $\Gamma$-space.\\

Alternatively, boundaries can be characterized in terms 
of their function algebras.
This is key in allowing the use of algebraic tools in the study of boundary actions.

The operator algebraic description of topological boundaries 
require some notions from the theory of injective envelopes 
as developed by Hamana~\cite{Hamana-85}. 
Since we will not use this in our work here, 
we only recall the main result regarding this 
characterization, and refer the reader to 
\cite{Kalantar-Kennedy-17, Breuillard-Kalantar-Kennedy-Ozawa-17} for more details.

\begin{thm}\cite[Theorem 3.11]{Kalantar-Kennedy-17}
Let $\Gamma$ be a discrete group. Then 
$C(\partial_F\Gamma)$ is the smallest 
injective object in the category of unital $\Gamma$-$C^*$-algebras.
\end{thm}

The $L^\infty$-algebras of measurable boundaries 
are precisely invariant von Neumann subalgebras of 
the algebra of bounded harmonic functions.

Recall for $\mu\in\pr(\Gamma)$ a 
function $f \in \ell^\infty(\Gamma)$ is said to be 
\define{$\mu$-harmonic} if
\begin{equation}\label{harm-def}
f(g) = \sum_{h\in\Gamma} \mu(h) f(g h) ~~\text{ for all } g \in \Gamma .
\end{equation}
We denote by $H^{\infty}(\Gamma,\mu) \subset \ell^\infty(\Gamma)$ the space of all 
bounded $\mu$-harmonic functions.
Observe that $H^{\infty}(\Gamma,\mu)$ is invariant under the action of $\Gamma$ by left translations. 

The space $H^{\infty}(\Gamma,\mu)$ is not a subalgebra of 
$\ell^\infty(\Gamma)$ in general, 
but the formula
\begin{equation}\label{harm-mult}
f_1\cdot f_2 (g) := \lim_{n\to\infty}\sum_{h\in \Gamma}f_1(h) f_2(h^{-1}g)\mu^n(h)
\end{equation}
defines a multiplication on $H^{\infty}(\Gamma,\mu)$ and turns it to a commutative von Neumann algebra.

\begin{prop}[Furstenberg]\label{prop:abstract-boundaries}
The Poisson map $\cP_{\nu_\infty}$ defines a von Neumann 
algebra isomorphism $L^\infty(\Pi_\mu, \nu_\infty)\cong H^\infty(\Gamma,\mu)$.

In particular, 
a measurable non-singular action $\Gamma\acts (Y, \eta)$ is a 
$\mu$-boundary
if and only if the Poisson map $\cP_\nu: L^\infty(X,\nu) \to H^\infty(\Gamma,\mu)$ 
is a von Neumann algebra embedding.
\end{prop}

Abstractly, we know where to find examples of 
boundary actions.
Measurable boundaries appear whenever one has a stationary action, 
and topological boundaries arise whenever one has an affine action on a compact convex space.

\begin{prop}[Furstenberg]
Suppose $(X, \nu)$ is a $(\Gamma, \mu)$-space. 
Then the weak* closure of the set of conditional measure 
$\{\nu_\om : \om\in\Omega\}$, with the push-forward of 
$\bP_\mu$ under the map $\om\mapsto\nu_\om$, 
is a $\mu$-boundary. Moreover every $\mu$-boundary 
arises in this way.
\end{prop}

\begin{prop}[\cite{Glasner-76}*{Theorem III.2.3}]\label{affine-act->top-bnd}
Suppose $\Gamma\acts K$ is an affine action, and suppose $K$ 
has no proper $\Gamma$-invariant compact convex subspace.
Then the closure of the extreme points of $K$ is a 
topological boundary.
Moreover every topological boundary of $\Gamma$ 
arises in this way.
\end{prop}

The main advantage of measurable boundaries 
over their topological counterparts is that they are much 
easier to concretely identify.
In fact there have been extensive work 
in the past few decades 
which have led to concrete realization of the Poisson boundary for 
many of groups that arise naturally as symmetries of 
geometric objects. We discuss some examples below.
In contrast, the Furstenberg boundary 
$\partial_F\Gamma$ of a non-amenable countable group $\Gamma$ 
is an extremally disconnected non-metrizable space, and 
not concretely realizable in any known case.

\subsection{USB systems}\label{sec:boudnaries}
The study of these systems was initiated by Furstenberg in~\cite{Furstenberg-73}, 
where they were called \emph{$\mu$-proximal actions}. They were further studied 
in~\cite{Margulis-book-91, Glasner-Weiss-16}.

\begin{defn}\label{defn:USB}
Let $\mu\in \pr(\Gamma)$. 
We say that a $\mu$-boundary $(X, \nu)$ is a 
($\mu$-)\define{USB} if it has 
a compact model $(K,\bar \nu)$ such that 
$\bar\nu$ is the unique $\mu$-stationary Borel probability measure 
on $K$. 
If $(X, \nu)$ is the ($\mu$-)Poisson boundary, we say that $(X,\nu)$ is a ($\mu$-)\define{Poisson-USB}.
\end{defn}

The following is a standard technique that allows us
to assume that the topological model of a USB is metrizable.
This will be important for us as we will make a heavy use of
the existence of conditional measures.
\begin{lemma}\label{lem:metrizable-model}
Every USB $(X, \nu)$ has 
a compact metrizable model $(K,\bar \nu)$ such that 
$\bar\nu$ is the unique $\mu$-stationary Borel probability measure 
on $K$.
\end{lemma}
\begin{proof}
Let $(K',\eta)$ be a compact model of $(X,\nu)$ as in the Definition~\ref{defn:USB}. 
As abstract probability spaces, $\mu$-boundaries are always standard. 
Therefore, using the fact that $\Gamma$ is countable we can choose
a countable, $\Gamma$-invariant set of elements
in $C(K')$ that includes the constant function, and is weak* dense in $L^\infty(K',\eta)$.
Take $K$ to be the spectrum of the ($\sup$-)norm closure of this set, which is compact and metrizable. Then $K$ is a continuous $\Gamma$-equivariant image of $K'$. Let $\bar\nu$ be the push-forward of $\eta$ to $K$. It follows (see Corollary~\ref{unique-stationary->subalg}) that $\bar\nu$ is the unique $\mu$-stationary probability on $K$, and so $(K,\bar \nu)$ is the desired compact metrizable model.
\end{proof}

Many natural examples of Poisson boundaries are in fact USB, 
namely, the Poisson boundary is being realized as a unique
stationary measure on a compact space.
The main tool for realizing the Poisson boundary
on a compact space is the strip criterion of Kaimanovich~\cite{Kaimanovich-00}, 
which proves, in many cases that the Poisson measure 
is actually unique. To name some examples (by no mean a complete list!) are linear groups acting
on flag varieties~\cite{Ledrappier-85, Kaimanovich-00, Brofferio-Schapira-11}, 
hyperbolic groups acting on the Gromov boundary~\cite{Kaimanovich-00}, non-elementary subgroups of mapping class groups acting on the
Thurston boundary~\cite{Kaimanovich-Masur-96}, and non-elementary 
subgroups of $\rm{Out}(\bF_n)$ acting on the boundary of the outer space~\cite{Horbez-16}.
We discuss properties of these actions further in Section~\ref{faithfulness-and-freeness-of-USB}.

\begin{theorem}\label{USB->top-bnd}
Let $\mu\in\pr(\Gamma)$, and suppose $(X, \nu)$ is a $\mu$-USB.
Then for any compact model of $(X, \nu)$, 
the restriction of the action to the support of the unique $\mu$-stationary 
is a topological boundary action.
\end{theorem}

\begin{proof}
To simplify notations, we assume $X$ is already a compact model, that is 
$X$ is a compact $\Gamma$-space and $\nu\in\pr(X)$ is the unique $\mu$-stationary 
measure such that $(X, \nu)$ is a $\mu$-boundary. 
Note that by unique stationarity of $\nu$, its support 
is the unique minimal component of $X$.
Thus, by passing to $\supp(\nu)$ if necessary, we also assume $\Gamma\acts X$ 
is minimal.

Assume first that $X$ is metrizable. We show $\pr(X)$ does not contain any proper 
non-empty compact convex $\Gamma$-invariant subsets. 
Then the theorem follows from Proposition~\ref{affine-act->top-bnd}.
Suppose $\cC\subseteq \pr(X)$ is a non-empty closed convex 
$\Gamma$-invariant set. 
Let $\eta\in \cC$. By $\Gamma$-invariance and closedness of $\cC$ we have
$\frac1n\sum_{k=0}^{n-1}\mu^k*\eta \in \cC$ for all $n\in\bN$.
Note any weak* cluster point of this sequence is $\mu$-stationary, hence
by uniqueness assumption $\nu\in \cC$.
As $X$ is metrizable, and by the closedness of $\cC$ we conclude that the
conditional measures, $\nu_\om \in \cC$
for $\bP_\mu$-a.e.\ path $\om\in \Omega$.
Since $(X,\nu)$ is $\mu$-boundary, $\nu_\om$ are point measures for $\bP_\mu$-a.e.\ $\om$.
In particular, there is some $x\in X$ such that $\delta_x\in \cC$. By $\Gamma$-invariance
of $\cC$ we get $\delta_{gx}\in \cC$ for every $g\in\Gamma$, and therefore minimality
of $X$ and closedness of $\cC$ yield $\{\delta_x :x\in X\} \subset \cC$. 
Now the convexity of $\cC$ implies $\pr(X) = \overline{\mathrm{conv}}\{\delta_x :x\in X\} \subset \cC$.

Now for the general case (not necessarily metrizable), 
let $\eta\in\pr(X)$ and $f\in C(X)$. Let $\cA\subset C(X)$ be the $\Gamma$-invariant 
$C^*$-subalgebra generated by $f$. Then $\cA = C(Y)$ for a metrizable $\Gamma$-factor 
of $X$. Moreover, the pushforward of $\nu$ on $Y$ is the unique $\mu$-stationary measure 
on $Y$ (see Corollary~\ref{unique-stationary->subalg} below), 
and thus by the above $\Gamma\acts Y$ is a topological boundary action. 
Therefore, we have $\|\cP_\eta(f)\| = \|f\|$. This shows the Poisson map $\cP_\eta$ is isometric. 
Hence, we conclude $X$ is a topological $\Gamma$-boundary.
\end{proof}

Perhaps the most significant application of topological 
boundaries so far has been in  
the problems of unique trace property and $C^*$-simplicity: 
the existence of a faithful topological boundary is
equivalent to the unique trace property and the existence
of a free boundary is equivalent to $C^*$-simplicity. 
A subtlety in applying these characterizations 
is that in general one has to pass to the maximal boundary, 
i.e.\ the Furstenberg boundary, which is too ``large'' 
to concretely realize and work with. 

But we will see, for example, that for determining 
the $C^*$-simplicity and the unique trace property of a group, it is
enough to work with its USB actions 
(provided that the group admits such), 
rather than abstract topological boundaries.

Recall that any group $\Gamma$ admits a maximal normal amenable subgroup, called the 
\define{amenable radical} of $\Gamma$. 
We denote this subgroup by $\Rad(\Gamma)$.

\begin{prop}\label{thm:kernel-of-USB}
Suppose $\Gamma \acts (X, \nu)$ is a $\mu$-USB 
for some $\mu\in\pr(\Gamma)$.
Then $\Rad(\Gamma) \subseteq \ker(\Gamma\acts(X,\nu))$,
and equality holds if $(X,\nu)$ is the $\mu$-Poisson USB.
\end{prop} 
\begin{proof}
The first assertion can be proven by 
a straightforward modification of the
proof of~\cite[Proposition 7]{Furman-03}.
Alternatively, by Theorem~\ref{USB->top-bnd} 
the action $\Gamma\acts \supp(\nu)$ is a topological boundary action, hence by the conclusion of~\cite[Proposition 7]{Furman-03}, $\Rad(\Gamma)$ acts trivially on $\supp(\nu)$, 
and so $\Rad(\Gamma)\subset \ker(\Gamma\acts(X,\nu))$.

For the second part of the statement, 
if $(X,\nu)$ is the Poisson boundary, then the action $\Gamma\acts (X,\nu)$ is 
Zimmer-amenable~\cite{Zimmer-78} and in particular the 
stabilizers $\stab_\Gamma(x)$ are amenable for $\nu$-a.e.\ $x\in X$. 
If follows that $\ker(\Gamma\acts(X,\nu))$ is a normal amenable group, 
and so $\ker(\Gamma\acts(X,\nu))\subset \Rad(\Gamma)$.
\end{proof}

\section{Stationary $C^*$-dynamical systems}\label{Non-commutative stationary dynamical systems}
Similarly to classical ergodic theory, invariant states may not exist 
in the setting of actions of non-amenable groups on $C^*$-algebras.
For instance, in the case of inner action by subgroups of the unitary group of a 
given $C^*$-algebra,
which is only non-trivial in the noncommutative setting,
the invariant ergodic theory is only available in the tracial case. 
Hence, one has to
appeal to other models of dynamical systems in infinite-type cases. 
In this section we begin studying 
the concept of stationary dynamical systems
in the context of $C^*$-algebras.\\

Let $\Gamma\acts\cA$ and let $\mu \in \pr(\Gamma)$. 
The $\mu$-convolution map on $\cA$ is defined by 
\[
\mu*a := \sum_{g\in \Gamma} \mu(g) g^{-1}a .
\]
Its adjoint
induces a $\mu$-convolution operator on the space of states 
$\cS(\cA)$ given by
\begin{align*}
\mu * \tau = \sum_{g\in \Gamma} \mu(g) g\tau .
\end{align*}

\begin{definition}
Let $\cA$ be a $\Gamma$-$C^{*}$-algebra. A state $\tau\in\cS (\cA)$ is 
said to be \define{$\mu$-stationary} if $\mu * \tau = \tau$. 
In this case we say the pair $(\cA,\tau)$ is a \define{$(\Gamma,\mu)$-$C^*$-algebra}.

We denote the collection of all $\mu$-stationary states on $\cA$ by $\cS_{\mu}(\cA)$.
We say that $\cA$ is \define{uniquely stationary} if there
exists $\mu \in \pr(\Gamma)$ such that $\cS_{\mu}(\cA)$ has only one element.
\end{definition}

\subsection{Basic facts}
In this section we review noncommutative versions of few basic facts about
stationary actions. 
First, note that invariant states are $\mu$-stationary 
for every $\mu\in\pr(\Gamma)$. Next, we observe that in contrast to the invariant case, stationary states always exist,
and moreover, they can always be extended. 
The corresponding statement of the latter in the commutative
setting is that 
any stationary measure on a factor can be ``pulled back'' (not necessarily in 
a unique way) to a stationary measure on the extension.

\begin{prop}\label{lem:extension-of-statioanry}
Suppose $\cA$ is a $\Gamma$-$C^{*}$-algebra
and $\cB \subset \cA$ is a $\Gamma$-invariant subalgebra.
Then every $\mu$-stationary $\eta\in\cS_\mu(\cB)$ can be extended to a $\mu$-stationary 
state $\tau\in\cS_\mu(\cA)$.

In particular, for any $\Gamma$-$C^*$-algebra $\cA$, and any $\mu\in\pr(\Gamma)$, the set $\cS_{\mu}(\cA)$ is non-empty.
\end{prop}
\begin{proof}
Let $E = \{ \rho \in \cS(\cA) : \rho|_\cB = \eta \}$. Then $E$ is
a compact convex subset of $\cS(\cA)$ and the convolution map by $\mu$ is an affine contraction on $E$.
Hence by
Tychonoff (or Kakutani) fixed point theorem there is $\tau\in E$ such that $\mu * \tau = \tau$.
\end{proof}

\begin{cor}\label{unique-stationary->subalg}
Let $\cA$ be a $\Gamma$-$C^*$-algebra, and let $\cB\le \cA$ be 
a $\Gamma$-invariant $C^*$-subalgebra. Suppose $\mu\in\pr(\Gamma)$
and $\tau\in\cS(\cA)$ is unique $\mu$-stationary.
Then $\tau|_\cB \in \cS(\cB)$ is unique $\mu$-stationary
for the action $\Gamma\acts \cB$. 
\end{cor}

Let $\cA$ be a $\Gamma$-$C^*$-algebra. The \define{Poisson map} $\cP_\tau: \cA \to \ll(\Gamma)$ associated
to a state $\tau$ on $\cA$ is defined by
\[
\cP_\tau(a)(g) = \langle g^{-1}a, \tau \rangle .
\]
Poisson maps are unital, positive and $\Gamma$-equivariant. 
We observe the converse.

\begin{lemma}
Suppose $\varphi: \cA \to \ll(\Gamma)$ is a unital positive $\Gamma$-equivariant
map. Then there is $\tau\in\cS(\cA)$ such that $\varphi = \cP_\tau$.
\end{lemma}

\begin{proof}
Suppose $\varphi$ is as above. Define the linear functional $\tau$ on $\cA$ by
$\langle a, \tau \rangle = \varphi(a)(e)$
for all $a\in \cA$. Since $\varphi$ is positive and unital, it follows $\tau$ is a state on $\cA$,
and moreover we have
\[
\cP_\tau(a)(g) = \langle g^{-1}a, \tau\rangle 
= \varphi(g^{-1}a)(e) = \left(g^{-1}(\varphi(a)\right)(e) = \varphi(a)(g)
\]
for all $g\in \Gamma$ and $a\in \cA$.
\end{proof}

\begin{lemma}\label{stationary<->Poisson-map-harmonic}
Suppose $\cA$ is a $\Gamma$-$C^*$-algebra, and let $\mu\in\pr(\Gamma)$.
Then a state $\tau\in\cS(\cA)$ is $\mu$-stationary if and only if $\cP_\tau(a)\in H^\infty(\Gamma, \mu)$ for every $a\in\cA$.
\end{lemma}

\begin{proof}
Suppose $\tau$ is $\mu$-stationary, then for every $a\in\cA$ and $g\in \Gamma$ we have
\begin{align*}
\sum_{h\in\Gamma} \mu(h) \cP_\tau(a)(gh) 
&= 
\sum_{h\in\Gamma} \mu(h) \langle h^{-1}g^{-1}a, \tau\rangle
\\&=
\langle g^{-1}a, \sum_{h\in\Gamma} \mu(h) h\tau\rangle
=
\langle g^{-1}a, \tau\rangle 
= \cP_\tau(a)(g),
\end{align*}
which shows $\cP_\tau(a) \in H^\infty(\Gamma,\mu)$.

Conversely, suppose $\cP_\tau(a)\in H^\infty(\Gamma, \mu)$ for all $a\in\cA$.
Then
\begin{align*}
\langle a, \sum_{h\in\Gamma} \mu(h) h\tau\rangle
&=
\sum_{h\in\Gamma} \mu(h) \cP_\tau(a)(h) 
= 
\cP_\tau(a)(e)
= 
\langle a, \tau\rangle 
\end{align*}
for all $a\in\cA$, which implies $\sum_{h\in\Gamma} \mu(h) h\tau = \tau$.
\end{proof}
The following fundamental result which is the noncommutative version of Theorem~\ref{cond-meas} follows from a more general result in the setting of affine pointed $(\Gamma, \mu)$-spaces in the sense of \cite[Theorem 2.16]{Bader-Shalom-06}.

We include a proof for the convenience of the reader. 
\begin{thm}\label{martingales}
Suppose $(\cA, \tau)$ is a separable $(\Gamma, \mu)$-$C^*$-algebra.
Then the weak* limits $\tau_\om:=\lim_n \om_n\tau$ exists for
$\bP_\mu$-a.e.\ path $\om\in\Omega$.
Moreover, we have 
\begin{equation}\label{disint}
\tau \,=\, \int_\Omega \tau_\om\, d\bP_\mu(\om)
\end{equation}
in the weak* sense. We call the states $\tau_\om$ \define{conditional states}.
\end{thm}

\begin{proof}
The function $f^a(g):=\langle a, g\tau\rangle$ is $\mu$-harmonic for every $a\in \cA$ by Lemma~\ref{stationary<->Poisson-map-harmonic}. Therefore, by the {martingale convergence theorem}, the limits $\bar{f}^a(\om):=\lim_n \langle a, \om_n\tau\rangle$ exists for $\bP_\mu$-a.e. $\om\in\Omega$. Furthermore, $\int_\Omega \bar{f}^a(\om)\, d\bP_\mu(\om) = \tau(a)$ (cf. \cite[Theorem 2.8]{Bader-Shalom-06}).

Now let $\cD$ be a countable dense subset of $\cA$. By the above, there is a set of full measure $\Omega_0$ such that $\bar{f}^{a}(\om)$
exists for every $a\in\cD$ and $\om\in\Omega_0$. By density of $\cD$, it follows that $\bar{f}^{a}(\om)$
exists for every $a\in\cA$ and $\om\in\Omega_0$.
Hence the conditional states
$\tau_\om (a):= \lim_n \om_n\tau(a) = \bar{f}^{a}(\om)$, $a\in \cA$, exist for every $\omega=(\om_n)\in \Omega_0$, and furthermore, \eqref{disint} holds.
\end{proof}

\subsection{Unique stationary actions}\label{Unique Stationary Actions} 
In classical dynamics unique ergodicity (i.e.\ existence of a unique invariant measure on a compact space)
is equivalent to uniform convergence of 
averages of continuous functions in the Birkhoff's ergodic theorem.
For general non-amenable groups, 
instead, the appropriate
notion is unique stationarity.
Glasner--Weiss~\cite{Glasner-Weiss-16} proved that a $(G,\mu)$-space $(X,\nu)$
is unique stationary if and only if for every $f\in C(X)$ the averages
of convolutions $\frac1n\sum_{k=0}^{n-1} \mu^k * f$ converge uniformly to $\int f d\nu$.

\begin{prop}\label{glasner-weiss}
An action $(\Gamma, \mu)\acts (\cA, \tau)$ is uniquely stationary if and only if 
\begin{equation}\label{lim}
\left\|\frac1n \sum_{k=0}^{n-1} \mu^k * a - \tau(a) 1_\cA\right\| \xrightarrow{n \to \infty} 0
\end{equation}
for all $a\in\cA$. 
\end{prop} 

\begin{proof}
Let $\mu_n = \frac1n \sum_{k=0}^{n-1} \mu^k$.
Suppose \eqref{lim} holds for all $a\in\cA$, 
and let $\eta\in \cS_\mu(\cA)$.
Then
\[
\langle a, \eta\rangle  = \langle a, \mu_n * \eta \rangle = 
\langle \mu_n * a, \eta \rangle 
\to \langle a, \tau \rangle \langle 1_\cA, \eta \rangle = \langle a, \tau \rangle
\]
for all $a\in\cA$. Hence $\eta = \tau$.

Conversely, suppose $\tau$ is the unique $\mu$-stationary state on $\cA$.
Observe that for every $a\in \cA$ we have $\left\|\mu_n \ast (\mu\ast a -a)\right\|\to 0$.
Thus, if we let $V_0 = \overline{\operatorname{span}}\{\mu * a - a : a\in\cA\}$,
then $\left\|\mu_n \ast b\right\|\to 0$ for all $b\in V_0$.
Also by stationarity, $\tau$ vanishes on $V_0$.
Consequently, for $t\in\bC$ and $b\in V_0$ we get
\begin{align*}
\left\|\mu_n \ast (t1_\cA - b) -  \langle t1_\cA - b, \tau \rangle 1_\cA\right\|
&=
\left\|t1_\cA - \mu_n \ast b -  \langle t1_\cA , \tau \rangle 1_\cA\right\|
\\&=
\left\|\mu_n \ast b\right\|
\to 0 ,
\end{align*}
which shows that \eqref{lim} holds for every $a\in V = \bC\oplus V_0$. 
Next, we show $V= \cA$ which then completes the proof of the theorem.
For this, let $\eta\in \cA^*$ with $\eta|_V=0$. 
Since $\eta$ vanishes on $V_0$ we see that $\eta$ is $\mu$-stationary. It follows from uniqueness of the Jordan decomposition $\eta = \eta_+-\eta_-$ that both the positive part $\eta_+$ and the negative part $\eta_-$ of $\eta$ are $\mu$-stationary, hence multiples of $\tau$ by unique stationarity. In particular, $\eta$ is a multiple of $\tau$, and since $\eta(1)=0$ we have $\eta=0$. Thus, it follows from the Hahn-Banach theorem that $V= \cA$.
\end{proof}

\subsection{Inner actions: stationary states as generalizations of traces}
An exclusive feature of noncommutative $C^*$-algebras is 
non-triviality of the inner action by their unitary groups. 
This allows one to consider a $C^*$-algebra $\cA$ as rather 
a $C^*$-dynamical system. 
In this point of view, traces on $\cA$ are nothing but 
invariant states, which may or may not exist in general. 
Thus, stationary states are generalizations of  
traces that do always exist, and in fact may be 
more appropriate objects to consider 
when the groups involved are non-amenable.

In this section, we prove some basic properties of 
stationary states in this setup, and see that 
they satisfy some useful properties of traces. 

\begin{lem}\label{kernel-of-stationary-is-ideal}
Suppose $\pi$ is a unitary representation of $\Gamma$, and 
consider the action $\Gamma\acts C^*_\pi(\Gamma)$ by 
inner automorphisms. 
Let $\mu\in\pr(\Gamma)$ be generating, and suppose $\tau$ 
is a $\mu$-stationary state on $C^*_\pi(\Gamma)$. 
Then the left kernel $I_\tau=\{a\in C^*_\pi(\Gamma) : \tau(a^*a) = 0 \}$ of $\tau$ 
is a two-sided closed ideal of $C^*_\pi(\Gamma)$.
\end{lem}

\begin{proof}
The inequality $a^*b^*ba\leq \|b^*b\|a^*a$ for operators on Hilbert spaces implies the well-known fact that the left kernel of any state is a left ideal. It is also obviously closed.
We show $I_\tau$ is also $\Gamma$-invariant. Let $a\in I_\tau$. Then 
\begin{align*}
\sum_{g\in \Gamma} \mu(g) \tau((\pi(g^{-1})a\pi(g))^*(\pi(g^{-1})a\pi(g)))
&=
\sum_{g\in \Gamma} \mu(g) \tau(\pi(g^{-1})a^*a\pi(g))
\\&= 
\tau(a^*a) 
=
0 , 
\end{align*}
which implies $\tau\big((\pi(g^{-1})a\pi(g))^*(\pi(g^{-1})a\pi(g))\big) = 0$ 
for every $g\in \supp(\mu)$. 
This implies $\pi(g^{-1})I_\tau\pi(g) \subset I_\tau$ for every $g\in \supp(\mu)$. 
Since $\mu$ is generating the same is true 
for every $g\in \Gamma$.
Thus, for every finite linear combination $b= \sum_{i=1}^n t_i \pi(g_i)\in C^*_\pi(\Gamma)$, 
$t_i\in \bC$, and every $a\in I_\tau$, there are $a_1,\dots, a_n\in I_\tau$ such that 
\[
a b = \sum_{i=1}^n t_i a\pi(g_i) =
\sum_{i=1}^n t_i \pi(g_i)a_i ,
\]
and the latter sum is in $I_\tau$ since it is a left ideal.
This shows $I_\tau$ is also a right ideal.
\end{proof}

Since every ideal of a $C^*$-algebra is invariant with respect to inner action by the 
unitary group, the problem of simplicity of a $C^*$-algebra 
translates into a minimality problem for a noncommutative 
dynamical system. Hence connection to stationarity is expected.

\begin{prop}\label{lem:non-simple-non-faithful}
Let $\pi$ be a unitary representation of $\Gamma$. 
The $C^*$-algebra $C^*_\pi(\Gamma)$ is simple 
if and only if there is a generating $\mu\in\pr(\Gamma)$ 
such that every $\mu$-stationary state on $C^*_\pi(\Gamma)$ 
is faithful.
\end{prop}

\begin{proof}
If $C^*_\pi(\Gamma)$ is simple and $\mu\in\pr(\Gamma)$ is generating, 
then by Lemma~\ref{kernel-of-stationary-is-ideal} every $\mu$-stationary 
state is faithful.

Conversely, suppose for some generating $\mu\in\pr(\Gamma)$, all 
$\mu$-stationary states are faithful.
Assume for sake of contradiction that $C^*_\pi(\Gamma)$ 
has a non-trivial proper ideal $I$. 
Since every ideal is invariant under the inner action of $\Gamma$, 
the action $\Gamma\acts C^*_\pi(\Gamma)$ induces an action $\Gamma\acts C^*_\pi(\Gamma) /I$. 
By Proposition~\ref{lem:extension-of-statioanry} there exists 
a $\mu$-stationary state $\tau$ on $C^*_\pi(\Gamma) /I$. Composing $\tau$
with the canonical quotient map $C^*_\pi(\Gamma)\to C^*_\pi(\Gamma) /I$ we obtain 
a $\mu$-stationary state on $C^*_\pi(\Gamma)$ that vanishes on $I$, 
which contradicts the assumption.
\end{proof}

We note that the condition of $\mu$ being generating 
is only needed for one side of the above. 
In fact, if $\cA$ is a unital $C^*$-algebra, and there is 
a countably supported probability measure $\mu$ on the 
unitary group $\cU(\cA)$ such that every $\mu$-stationary state on 
$\cA$ is faithful, then $\cA$ is simple. \\

An important special case is the reduced $C^*$-algebra.
\begin{cor}\label{unique-trace->simple}
If $\Gamma\acts\csr$ is uniquely stationary then $\Gamma$ is $C^*$-simple.
\end{cor}
\begin{proof}
The canonical trace $\tau_0$ is $\Gamma$-invariant, and hence 
$\mu$-stationary for any $\mu$. Recall also that $\tau_0$ is faithful. The corollary
now follows from Proposition~\ref{lem:non-simple-non-faithful}.
\end{proof}

In Section~\ref{An application: a new characterization of $C^*$-simplicity} 
we will prove the converse of this, which provides a new characterization of $C^*$-simplicity.
But at this point some concrete examples are in order. In particular, we 
demonstrate how in general unique stationarity can be deduced. 

Here and throughout the paper, $\bF_2$ denotes the free group
on two generators $a$ and $b$, and 
$\partial \bF_2$ denotes its Gromov boundary, which is 
a compact space naturally identified 
with the set of all infinite reduced words in the generators.
We have the natural action $\bF_2\acts \partial \bF_2$ 
by concatenation with the subsequent cancellation of pairs 
of consecutive inverses.

\begin{example}\label{ex:SRW-on-free-group}
Let $\mu\in\pr(\bF_2)$ be the uniform measure on the set of generators $\{a, a^{-1}, b, b^{-1}\}$.
We show the canonical trace $\tau_0$ is the unique
$\mu$-stationary state on $C^*_\lambda(\bF_2)$.
One can see that 
the ``uniform measure'' on $\partial \bF_2$, 
given by $\nu ([w])=\frac{1}{4\cdot3^{n-1}}$ where $w$ is a finite word of length $n$ 
and $[w]$ is the set of all infinite reduced words that start with $w$, 
is the unique $\mu$-stationary probability on $\partial \bF_2$. 
Moreover, $(\partial \bF_2, \nu)$ is a $\mu$-boundary. 

Now, let $\tau$ be a $\mu$-stationary state on $C^*_\lambda(\bF_2)$.
By Proposition~\ref{lem:extension-of-statioanry} we can extend $\tau$ to a $\mu$-stationary state
$\tilde\tau$ on $\bF_2\ltimes_r C(\partial \bF_2)$, where $\bF_2\acts \bF_2\ltimes_r C(\partial \bF_2)$ is
also by inner automorphisms.
Then $\tilde\tau|_{C(\partial \bF_2)}$ is stationary and 
by uniqueness, this restriction is $\nu$. Hence
${\tilde\tau}_\om|_{C(\partial \bF_2)} = \delta_{\bnd(\om)}$
for a.e.\ $\om\in \Omega$, where $\bnd: (\Omega,\bP_\mu)\to (\partial \bF_2,\nu)$ is 
the boundary map.

It is obvious that the action $\bF_2 \acts (\partial\bF_2,\nu)$ is essentially free.
Hence, it follows from Lemma~\ref{lemma0}
that for every non-trivial $g\in\bF_2$, $\tau_\om(\lambda_g)=0$
for $\bP_\mu$-a.e.\ $\om$. Thus, $\tau_\om = \tau_0$ for $\bP_\mu$-a.e.\ path $\om\in\Omega$.
Thus, applying Theorem~\ref{martingales} we get
\[
\tau = \int_\Omega \tau_\om\, d\bP_\mu(\om) = \tau_0,
\]
which shows $\tau_0$ is the unique $\mu$-stationary state on $C^*_\lambda(\bF_2)$.
\end{example}
We note that in the above reasoning there is nothing particularly special about the uniform  measure. The above conclusion holds for any generating measure
on $\bF_2$. In fact, there is nothing also particularly special about 
the free group here, the conclusion holds for any measure $\mu$
on any groups $\Gamma$ that admits an essentially free $\mu$-USB. 

\begin{thm}\label{free-USB->C*-simple} 
Suppose $\Gamma$ is a countable discrete group, and let 
$\mu\in\pr(\Gamma)$. If $\Gamma$ admits
an essentially free $\mu$-USB, then the canonical trace $\tau_0$ on $C^*_\lambda(\Gamma)$ is uniquely
$\mu$-stationary.
\end{thm}

\begin{proof}
Repeat the argument given in Example~\ref{ex:SRW-on-free-group} above
on a metrizable model of the USB 
(such a model exists by Lemma~\ref{lem:metrizable-model}).
\end{proof}

In particular, any such group $\Gamma$ is $C^*$-simple.
Of course if $(X, \nu)$ is 
an essentially free USB, by Theorem~\ref{USB->top-bnd} 
the action $\Gamma\acts\supp(\nu)$ 
is a topologically free topological boundary, 
hence $\Gamma$ is $C^*$-simple 
by results of~\cite{Breuillard-Kalantar-Kennedy-Ozawa-17}.
However, for those groups with essentially free USB actions, 
the above theorem, besides 
giving a much simpler proof of $C^*$-simplicity, 
reveals more than just $C^*$-simplicity of $\Gamma$, 
namely, a probability 
$\mu\in\pr(\Gamma)$ with respect to 
which the canonical trace is uniquely stationary.
In Section~\ref{An application: a new characterization of $C^*$-simplicity} 
we will prove that the existence of such $\mu$ is 
equivalent to $C^*$-simplicity of $\Gamma$,
and we conclude various properties of boundaries and 
random subgroups associated to such measures. \\

Next is an example of a faithful uniquely stationary state 
on a purely infinite $C^*$-algebra.

\begin{example}\label{pure-inf}
Let $\cA = \bF_2\ltimes_r C(\partial \bF_2)$, and 
let $\mu\in\pr(\bF_2)$ be generating. As usual, let $\tau_0$ denote the
canonical trace on $C^*_\lambda(\bF_2)$. Also, let $\nu\in\pr(\partial\bF_2)$
be the unique $\mu$-stationary probability on $\partial\bF_2$.
Now suppose $\tau\in \cS_\mu(\cA)$. Then $\tau|_{C(\partial\bF_2)} = \nu$,
and by Example~\ref{ex:SRW-on-free-group} above, $\tau|_{C^*_\lambda(\bF_2)} = \tau_0$.
Hence, we have $\tau_\om|_{C(\partial\bF_2)} = \delta_{\bnd(\om)}$ and  $\tau_\om|_{C^*_\lambda(\bF_2)} = \tau_0$ 
for a.e.\ path $\om\in\Omega$. 
So for a linear combination $\sum_{g\in\bF_2} f_g \lambda_g$ 
where $f_g\in C(\partial\bF_2)$ is non-zero 
for at most finitely many $g\in \Gamma$,
using the fact that $\delta_{\bnd(\om)}$ is multiplicative 
on $C(\partial\bF_2)$, we see
\begin{align*}
\langle\, \sum_{g\in\bF_2} f_g \lambda_g \,,\, \tau_\om\,\rangle 
&= 
\sum_{g\in\bF_2} f_g(\bnd(\om)) \langle\, \lambda_g \,,\,\tau_\om\,\rangle
\\&= 
\sum_{g\in\bF_2} f_g(\bnd(\om)) \tau_0(\lambda_g)
\\&= 
f_e(\bnd(\om)) .
\end{align*}
for a.e.\ path $\om\in\Omega$.
Thus, from \eqref{disint} it follows  
\begin{align*}
\langle\, \sum_{g\in\bF_2} f_g \lambda_g \,,\, \tau\,\rangle 
&= \int_{\Omega} \langle\, \sum_{g\in\bF_2} f_g \lambda_g \,,\, \tau_\om\,\rangle d\bP_\mu(\om)
\\&= \int_{\Omega} f_e(\bnd(\om)) d\bP_\mu(\om) 
\\&= \int_{\partial\bF_2}f_e d\nu .
\end{align*}
Since the set of all finite linear combinations $\sum_{g\in\bF_2} f_g \lambda_g$
is dense in $\cA$, this formula uniquely determines $\tau$. 
Note also $\tau = \nu\circ \bE$, where $\bE:\cA\to C(\partial\bF_2)$
is the canonical conditional expectation
$\sum_{g\in\bF_2} f_g \lambda_g \mapsto f_e$ (see e.g.~\cite{Brown-Ozawa-08}).
Since both $\bE$ and $\nu = \tau|_{C(\partial\bF_2)}$ are faithful, so is $\tau$. 
In particular, this also implies the well-known fact that 
$\bF_2 \ltimes_r C(\partial\bF_2)$ is simple (see e.g.~\cite{Archbold-Spielberg-94}).
\end{example}

Similarly, normal stationary states with respect to 
the inner action by the unitary group of a von Neumann 
algebra can provide a suitable replacement for normal traces 
in the case of non-finite von Neumann algebras.

\begin{prop}
Let $M$ be a von Neumann algebra and $\Gamma$ a group of unitaries in $M$.
Suppose for the action $\Gamma\acts M$ by inner automorphisms, 
$M$ admits a faithful unique normal $\mu$-stationary state $\tau$ for some 
$\mu\in\pr(\Gamma)$. Then $M$ is a factor.
\end{prop}

\begin{proof}
Suppose $M$ is not a factor, 
and let $p\in M$ be a non-trivial central projection.
Let $\mu\in\pr(\Gamma)$ and suppose $\tau\in M_*$ is a faithful normal 
$\mu$-stationary state. Set $\tau_1 := \frac{1}{\tau(p)} \tau(\cdot p)$. 
Then $\tau_1\in M_*$ is a normal state and using the fact that 
every central element is fixed by $\Gamma$, we get
\begin{align*}
\sum_{g\in\Gamma} \mu(g) \tau_1(g^{-1} a) &= 
\frac{1}{\tau(p)}\sum_{g\in\Gamma} \mu(g) \tau((g^{-1} a)p) \\&= 
\frac{1}{\tau(p)}\sum_{g\in\Gamma} \mu(g) \tau(g^{-1} (ap)) \\&= 
\frac{1}{\tau(p)}\tau(ap) \\&= 
\tau_1(a)
\end{align*}
for all $a\in M$, which shows $\tau_1$ is $\mu$-stationary.
Similarly, $\tau_2 := \frac{1}{\tau(1-p)} \tau(\cdot (1-p))$ is 
a normal $\mu$-stationary state, and obviously 
$\tau = \tau(p) \tau_1 + \tau(1-p) \tau_2$. But since $\tau_1(1-p) = 0$,
we have $\tau\neq\tau_1$, hence $\tau$ is not the unique normal $\mu$-stationary state
on $M$.
\end{proof}

\begin{example}\label{type-III}
We follow the notations of Example~\ref{pure-inf}.
The von Neumann algebra crossed product $M = \bF_2 \ltimes L^\infty(\partial\bF_2, \nu)$ 
is a type III factor.
The set of all linear combinations $\sum_{g\in\bF_2} f_g \lambda_g$,
where $f_g\in L^\infty(\partial\bF_2, \nu)$ is non-zero for at most finitely many $g\in\bF_2$, 
is weak* dense in $M$.
The map $\sum_{g\in\bF_2} f_g \lambda_g\mapsto f_e$ 
extends to a faithful normal conditional expectation $\bE : M\to L^\infty(\partial\bF_2, \nu)$. 
Thus $\sum_{g\in\bF_2} f_g \lambda_g\mapsto \int_{\partial\bF_2}f_e d\nu$ 
defines a faithful normal state $\tau$ on $M$.
Note that the reduced $C^*$-crossed product $\cA = \bF_2\ltimes_r C(\partial\bF_2)$
is an $\bF_2$-invariant weak* dense $C^*$-subalgebra.
From Example~\ref{pure-inf} we know $\tau|_\cA$ is $\mu$-stationary, 
where $\mu\in\pr(\bF_2)$ is any generating probability.
Since $\tau$ is normal and $\cA$ is weak* dense in $M$, it follows 
$\tau$ is a $\mu$-stationary state on $M$.
Moreover, if $\tau'$ is another normal $\mu$-stationary state on $M$, 
then its restriction $\tau'|_\cA$ is again $\mu$-stationary, hence equal to $\tau|_\cA$ 
by unique stationarity property established in Example~\ref{pure-inf}.
Since both $\tau$ and $\tau'$ are normal and $\cA$ is weak* dense in $M$, it follows 
$\tau'=\tau$, which implies unique stationarity of $\tau$.
\end{example}

\begin{example}[Noncommutative USB]\label{noncommutative USB}
Noncommutative Poisson boundaries were defined by Izumi in \cite{Izumi-02}.
This concept has found many important applications in various 
operator algebraic contexts.
Let us briefly recall the definition. Suppose $M$ is a von Neumann 
algebra, and $\Phi:M\to M$ is a Markov operator, i.e.\ a unital 
completely positive normal map. Then  
the fixed point space $\fix(\Phi)=\{x\in M : \Phi(x)=x\}$ is a 
unital self-adjoint weak* closed subspace of $M$, and there is 
a positive contractive idempotent $E:M\to \fix(\Phi)$. 
Endowed with the \emph{Choi-Effros product}, $x\circ y := E(xy)$, 
the space $\fix(\Phi)$ becomes a von Neumann algebra, called 
the Poisson boundary of $\Phi$, and denoted by $H_\infty(M, \Phi)$.

A class of examples of Markov operators are obtained 
from canonical extensions of convolution operators, as follows.
Let $\mu\in\pr(\Gamma)$, and define $\Phi_\mu:\cB(\ell^2(\Gamma))\to\cB(\ell^2(\Gamma))$ 
by $\Phi_\mu(x) = \sum_{g\in \Gamma} \mu(g) \rho_g x \rho_{g^{-1}}$, where 
$\rho:\Gamma\to \cU(\ell^2(\Gamma))$ is the right regular representation. 
Then $\Phi_\mu$ is a Markov map on $\cB(\ell^2(\Gamma))$, and Izumi 
proved \cite{Izumi-04} that the Poisson boundary 
$H_\infty(\cB(\ell^2(\Gamma)), \Phi_\mu)$ is canonically isomorphic to 
the von Neumann crossed product $\Gamma\ltimes H_\infty(\Gamma, \mu)$.

Now, for instance, continuing to follow the notations of Example~\ref{pure-inf}, 
$M = \bF_2 \ltimes L^\infty(\partial\bF_2, \nu)$ is identified with 
the Poisson boundary of the Markov map $\Phi_\mu$ where 
$\mu$ is the uniform measure on the set of generators.
Note also that $M$ is the 
von Neumann algebra generated by the reduced crossed product 
$\cA = \bF_2 \ltimes_r C(\partial\bF_2)$ in the GNS representation of 
the unique stationary state $\tau\in \cS(\cA)$. 
Thus, the $C^*$-dynamical system 
$\Gamma \acts (\cA, \tau)$ gives an example of a \emph{noncommutative Poisson USB}.
\end{example}

\begin{example}
Suppose $\pi: C^*_\lambda(\bF_2)\to \cB(\cH_\pi)$ is an irreducible representation. 
Consider the inner action $\Gamma\acts \cB(\cH_\pi)$ by unitaries $\pi(\lambda_g)$, $g\in\Gamma$. 
Let $\mu\in \pr(\bF_2)$ be generating, 
in particular $\supp(\mu)'' = \cB(\cH_\pi)$. 
We show in this case $\cB(\cH_\pi)$ does not admit any normal $\mu$-stationary 
state. 
Note first that simplicity of 
$C^*_\lambda(\bF_2)$ implies $\pi$ is injective.
Suppose $\tau\in\cB(\cH_\pi)_*$ is a $\mu$-stationary state, 
then $\tau|_{\pi(C^*_\lambda(\bF_2))}$ 
is $\mu$-stationary, hence equals to the canonical trace $\tau_0$ 
by Example~\ref{ex:SRW-on-free-group}. 
Since $\tau$ is normal, and tracial on a weak* dense subalgebra, 
it follows $\tau$ is a normal trace on $\cB(\cH_\pi)$, which implies $\cH_\pi$ 
is finite dimensional. But that cannot be the case since 
$\pi: C^*_\lambda(\bF_2)\to \cB(\cH_\pi)$ is injective.  
\end{example}

\section{A new characterization of $C^{*}$-simplicity}\label{An application: a new characterization of $C^*$-simplicity}

In this section we prove a new characterization of $C^*$-simplicity of a group $\Gamma$ in terms of unique stationarity of the action $\Gamma\acts \csr$. 

The following more general statement makes the connection between these properties more explicit.

\begin{thm}\label{new-characterization-C*-simplicity-general}
Let $\cA$ be a separable unital $\Gamma$-$C^*$-algebra with a $\Gamma$-invariant state $\tau$. Then the following are equivalent:

\begin{itemize}
\item[(i)]
The triple $(\Gamma, \cA, \tau)$ satisfies the Powers property: for every $x\in \cA$ and $\ep>0$, there exists $\mu_0\in \pr(\Gamma)$ such that $\|\mu_0*x - \tau(x)1\| <\ep$.
\item[(ii)]
There exists $\mu\in\pr(\Gamma)$ such that $\tau$ is the unique $\mu$-stationary state on $\cA$.
\end{itemize}
\end{thm}

\begin{proof}
The implication (ii)$\implies$(i) follows from Proposition~\ref{glasner-weiss}. 

The proof of (i)$\implies$(ii) follows a similar construction as in Kaimanovich--Vershik's proof (\cite[Theorem 4.3]{Kai-Ver}) of Furstenberg's conjecture.

Assume (i) holds. Choose an increasing sequence $\{n_k\}$ of positive integers such that
$(\sum_{i=1} ^{k}\frac{1}{2^i})^{n_k} < \frac{1}{2^k}$ for all $k\in\bN$.
Let $\{a_i\}_{i\in\bN}$ be a dense subset of the unit ball of $\cA$.
Using the above, for every $l\in \bN$ choose $\mu_{l}$, inductively, so that
\[
\left\| \mu_{l} * \mu_{k_r} * \cdots * \mu_{k_1} * a_s  - \tau(a_s)\, \mathds{1}_{\cA} \right\|
 < \frac{1}{2^l}
\]
for all $1\le s, k_1, \dots, k_r < l$, and $r < n_l$.
Let $\displaystyle\mu =  \sum_{l=1}^{\infty} \frac{1}{2^l} \mu_l \in \ell^1(\Gamma)$. 
Given any $a$ in the unit ball of $\cA$ and $\varepsilon > 0$, let $j\in\bN$ be such that
$\left\| a - a_j \right\| < \varepsilon$ and ${1}/{2^j} < \varepsilon$. 
Then 
\begin{eqnarray*}
\left\| \mu^{n_j} * a -  \tau(a) \, \mathds{1}_{\cA} \right\| &\leq& \big\| \mu^{n_j} * a - \mu^{n_j} * a_j \big\|
 + \left\| \mu^{n_j} * a_j -  \tau(a_j) \, \mathds{1}_{\cA} \right\| 
 \\&+& \left\|  \tau(a_j) \,\mathds{1}_{\cA} -  \tau(a) \,\mathds{1}_{\cA} \right\|\\
& <& 2\varepsilon + \left\| \mu^{n_j} * a_j -  \tau(a_j) \,\mathds{1}_{\cA} \right\|.
\end{eqnarray*}
We expand and split the term $\left\| \mu^{n_j} * a_j -  \tau(a_j) \,\mathds{1}_{\cA} \right\|$ as follows:
\begin{eqnarray*}
&&\left\| \,\sum_{\max k_i\leq j} \frac{1}{2^{k_{n_j}}\dots2^{k_1}}\,
\mu_{k_{n_j}} *  \cdots  * \mu_{k_1} * a_j
 + \sum_{\max k_i > j} \frac{1}{2^{k_{n_j}}\dots2^{k_1}}\, \mu_{k_{n_j}} *  \cdots  * \mu_{k_1}  * a_j\right.\\
&  & \left.- \left(\sum_{\max k_i\leq j} \frac{1}{2^{k_{n_j}}\dots2^{k_1}}\right)  \tau(a_j) \,\mathds{1}_{\cA}
 - \left(\sum_{\max k_i > j} \frac{1}{2^{k_{n_j}}\dots2^{k_1}}\right)  \tau(a_j) \,\mathds{1}_{\cA}
\,\right\| \\
&\leq& 
\left\| \sum_{\max k_i\leq j} \frac{1}{2^{k_{n_j}}\dots2^{k_1}}\,\left(
\mu_{k_{n_j}} *  \cdots * \mu_{k_1} * a_j -  \tau(a_j) \,\mathds{1}_{\cA}\right)\right\|
 \\&+&
  \left\|\sum_{\max k_i > j} \frac{1}{2^{k_{n_j}}\dots2^{k_1}}\,\left(
\mu_{k_{n_j}} *  \cdots * \mu_{k_1} * a_j -  \tau(a_j) \,\mathds{1}_{\cA}\right)\right\|
\\&\leq& 
2\|a_j\| \sum_{\max k_i\leq j}  \frac{1}{2^{k_{n_j}}\dots2^{k_1}} +
 \sum_{\max k_i > j} \frac{1}{2^{k_{n_j}}\dots2^{k_1}}  \left\| \mu_{k_{n_j}} *  \cdots * \mu_{k_1} * a_j -  \tau(a_j) \,\mathds{1}_{\cA} \right\|
\\&=& 
2\|a_j\| (\sum_{i=1} ^{j}\frac{1}{2^i})^{n_j} +
 \sum_{\max k_i > j} \frac{1}{2^{k_{n_j}}\dots2^{k_1}}  \left\| \mu_{k_{n_j}} *  \cdots * \mu_{k_1} * a_j -  \tau(a_j) \,\mathds{1}_{\cA} \right\|
\\&\leq& 2\varepsilon +
 \sum_{\max k_i > j} \frac{1}{2^{k_{n_j}}\dots2^{k_1}} \left\| \mu_{k_{n_j}} *  \cdots * \mu_{k_1} * a_j -  \tau(a_j) \,\mathds{1}_{\cA} \right\|.
\end{eqnarray*}
Now consider one of the terms
 $\mu_{k_{n_j}} *  \cdots * \mu_{k_1} * a_j$ in the last sum above
 and let $k_j$ be the first index such that $k_j > j$.
Then we have
\begin{eqnarray*}
&& \left\| \mu_{k_{n_j}} * \cdots * \mu_{k_1} * a_j -  \tau(a_j) \,\mathds{1}_{\cA}\right \| 
\leq
\left\| \mu_{k_j} *  \cdots * \mu_{k_1} * a_j -  \tau(a_j) \,\mathds{1}_{\cA} \right\| 
 < \varepsilon,
\end{eqnarray*}
where the last inequality follows from the construction of $\{\mu_l\}$. This implies
\begin{eqnarray*}
\sum_{\max k_i > j} \frac{1}{2^{k_{n_j}}\dots2^{k_1}} \left\| \mu_{k_{n_j}} *  \cdots * \mu_{k_1} * a_j -  \tau(a_j) \,\mathds{1}_{\cA} \right\|
 < \varepsilon .
\end{eqnarray*}
Hence we get $\left\| \mu^{n_j} * a -  \tau(a) \, \mathds{1}_{\cA} \right\| < 5\varepsilon$.
Since $\|\mu\| = 1$, this yields
\begin{align*}
\left\| \mu^{n} * a -  \tau(a) \, \mathds{1}_{\cA} \right\|
&=
\left\| \mu^{n-n_j} *\mu^{n_j} * a -  \tau(a) \, \mathds{1}_{\cA} \right\|
\\&\leq 
\left\| \mu^{n_j} * a -  \tau(a) \, \mathds{1}_{\cA} \right\|
\\&< 5\varepsilon
\end{align*}
for all  $n>n_j$.
Hence 
\[
\left\| \mu^{n} * a -  \tau(a) \, \mathds{1}_{\cA} \right\| \xrightarrow{n\to\infty} 0 
\]
for all $a\in \cA$, which by Proposition~\ref{glasner-weiss} implies that $\tau$
is the unique $\mu$-stationary state on $\cA$.
\end{proof}

Now, combining Theorem~\ref{new-characterization-C*-simplicity-general} with Haagerup's characterization of $C^*$-simplicity in~\cite[Theorem 4.5]{Haagerup-C*-simple-15}, we get the following.

\begin{thm}\label{new-characterization-C*-simplicity}
A countable discrete group $\Gamma$ is $C^*$-simple 
if and only if 
there is $\mu\in\pr(\Gamma)$
such that the canonical trace $\tau_0$
is the unique $\mu$-stationary state on $\csr$ with respect to the $\Gamma$-action by inner automorphisms. 
\end{thm}

\begin{proof}
Assume $\Gamma$ is $C^*$-simple. Let $f= \sum_{g\in\Gamma} f(g)\delta_g$ be a function on $\Gamma$ with finite support, and fix $\varepsilon_0>0$. We denote by $\lambda(f) = \sum_{g\in\Gamma} f(g)\lambda_g$
the \emph{left regular representation} of $f$.
By~\cite[Theorem 4.5]{Haagerup-C*-simple-15}
there are $h_1, h_2, \dots, h_n\in \Gamma$ such that
\[
\left\|\frac1n \sum_{k=1}^n \lambda_{h_k^{-1} g h_k}\right\| < \varepsilon_0
\]
for all $g\in \supp{f}\setminus \{e\}$. We then have
\begin{align*}
\left\|\frac1n \sum_{k=1}^n \lambda_{h_k^{-1}} \lambda(f) \lambda_{h_k} - \tau_0(\lambda(f))\mathds{1}_{\csr}\right\|
&=
\left\|\sum_{g\in\Gamma}\frac1n \sum_{k=1}^n f(g)\lambda_{h_k^{-1}gh_k} - \tau_0(\lambda(f))\lambda_e\right\|
\\ &\leq 
\sum_{g\neq e}\left\|\frac1n \sum_{k=1}^n f(g)\lambda_{h_k^{-1}gh_k}\right\|
\\ &+ 
\frac1n \left\|\big(f(e) - \tau_0(\lambda(f))\big)\lambda_e\right\|
\\&=
\sum_{g\neq e}|f(g)|\left\|\frac1n \sum_{k=1}^n \lambda_{h_k^{-1}gh_k}\right\|
\\&\leq
\left\|f\right\|_\infty \sum_{g\in \supp{f}\setminus \{e\}}\left\|\frac1n \sum_{k=1}^n \lambda_{h_k^{-1}gh_k}\right\|
\\&< \left\|f\right\|_\infty \#\{\supp{f}\}\varepsilon_0 .
\end{align*}
Now, let finitely supported functions $f_1, f_2, \dots, f_j$ on $\Gamma$, and 
$\ep>0$ be given. 
Let $F = \displaystyle \bigcup\limits_{i=1}^j \supp{f_i}$, and $c = \max_i\{\|f_i\|_\infty\}$. 
Then, setting $\ep_0 = \displaystyle \frac{\ep}{c\, (\#F)}$ in the above calculations, 
there are $h_1, h_2, \dots, h_n\in \Gamma$ such that for 
$\mu_0 = \displaystyle \frac1n \sum_{k=1}^n \delta_h \in \pr(\Gamma)$ we have 
\[
\left\|\mu_0 * \lambda(f_i)  - \tau_0(\lambda(f_i))\mathds{1}_{\csr}\right\| < \ep
\]
for all $i = 1, \dots, j$. 
Since the set $\{\lambda(f) : f \text{ has finite support}\}$ is norm-dense
in $\csr$, for any given $a_1, \dots, a_j \in \csr$ and $\ep>0$, we may find 
$\mu_0\in\pr(\Gamma)$ such that 
\[
\left\|\mu_0 * a_i  - \tau_0(a_i)\mathds{1}_{\csr}\right\| < \ep
\]
for all $i = 1, \dots, j$. 
Hence, by Theorem~\ref{new-characterization-C*-simplicity-general}, there is $\mu\in\pr(\Gamma)$ such that $\tau_0$ is the unique $\mu$-stationary state on $\csr$.

The converse follows from Corollary~\ref{unique-trace->simple}.
\end{proof}
{\vskip 0.3cm}

\subsection{$C^*$-simple measures}\label{new-characterization-C*-simplicity-applications}
In this section we prove several properties of the measures $\mu$ that 
``capture'' $C^*$-simplicity in the sense of 
Theorem~\ref{new-characterization-C*-simplicity}. 
We see that these measures 
posses significant 
ergodic theoretical properties.
Thus, in some sense, one should consider 
$C^*$-simplicity of $\Gamma$ as a property of the measure(s) 
$\mu\in \pr(\Gamma)$ in Theorem~\ref{new-characterization-C*-simplicity}.
\begin{defn}\label{def:c*-simple-measure}
We say that a measure  $\mu\in \pr(\Gamma)$
is \define{$C^*$-simple} if the canonical trace $\tau_0$ is the unique $\mu$-stationary state on $\csr$.
\end{defn}

\begin{thm}\label{Zimmer-amenable-action-C*-Simple-group-is-free}
Suppose $\mu\in\pr(\Gamma)$ is $C^*$-simple. 
Then any measurable $\mu$-stationary action 
with almost surely amenable stabilizers, is essentially free.

In particular, any Zimmer-amenable $\mu$-stationary action 
(e.g.\ the Poisson boundary action $\Gamma\acts(\Pi_\mu, \nu_\infty)$) is essentially free.
\end{thm}

\begin{proof}
Suppose $\mu\in\pr(\Gamma)$ is $C^*$-simple,
and let $\Gamma\acts (X,\nu)$ be a 
measurable $\mu$-stationary action
such that $\stab_\Gamma(x)$ is amenable for $\nu$-almost every $x\in X$. 

Consider the map $\Psi: X \to \Sub(\Gamma)$ defined by $\Psi(x) = \stab_\Gamma(x)$.
Then $\eta = \Psi_*\nu$ is an amenable $\mu$-SRS of $\Gamma$.
Therefore by Lemma~\ref{random-subgrp-->pos-def} there is 
a state $\tau$ on $\csr$ such that 
$\tau(\lambda_g) = \eta(\{\Lambda : g\in \Lambda\})$ for all $g\in\Gamma$.
Since $\eta(\{\Lambda : g\in \Lambda\}) = \nu(\{x : g\in \stab_\Gamma(x)\}) = \nu(\mathrm{Fix}(g))$, and 
since $\nu$ is $\mu$-stationary it follows 
\begin{align*}
\sum_{g\in\Gamma} \mu(g) \tau(\lambda_{g^{-1}}\lambda_h\lambda_g)
&= 
\sum_{g\in\Gamma} \mu(g) \nu(\mathrm{Fix}(g^{-1}hg))
\\&= 
\sum_{g\in\Gamma} \mu(g) \nu(g^{-1}\mathrm{Fix}(h))
\\&= 
\sum_{g\in\Gamma} \mu(g) g\nu(\mathrm{Fix}(h))
\\&=
\nu(\mathrm{Fix}(h))
\\&=
\tau(\lambda_h),
\end{align*}
which shows
$\tau$ is $\mu$-stationary. Hence, 
$\tau = \tau_0$, which yields $\nu(\mathrm{Fix}(g)) = 0$ for all 
non-trivial $g\in\Gamma$, i.e.\ $\Gamma\acts (X,\nu)$ 
is essentially free.
\end{proof}

Essential freeness of the abstract
Poisson boundary has some ergodic theoretical consequences, for example it implies genericity of stationary measures (see~\cite{Bowen-Hartman-Tamuz-17}).\\

The following corollary is a weaker conclusion of 
Theorem~\ref{Zimmer-amenable-action-C*-Simple-group-is-free} 
at the topological level. 
\begin{cor}{(see also 
\cite[Proposition 7.6 \& Remark 7.7]{Breuillard-Kalantar-Kennedy-Ozawa-17})}
Suppose $\Gamma$ is $C^*$-simple.
Then any minimal action $\Gamma\acts X$ on a compact space 
with amenable stabilizers is topologically free, that is, the set 
$\{x\in X : \stab_\Gamma(x) \text{ is trivial}\}$ is dense in $X$.
\end{cor}

\begin{proof}
Since by Proposition~\ref{lem:extension-of-statioanry} every 
compact $\Gamma$-space admits a stationary measure, it 
follows from Theorem~\ref{Zimmer-amenable-action-C*-Simple-group-is-free} 
that there is some $x\in X$ which has trivial stabilizer, and so does 
every point in its orbit. Now the assertion follows from minimality.
\end{proof}

Also, Theorem~\ref{Zimmer-amenable-action-C*-Simple-group-is-free} and Theorem~\ref{free-USB->C*-simple} 
imply the following measurable 
version of the main result of 
\cite{Kalantar-Kennedy-17}.

\begin{cor}
If $\Gamma$ is $C^*$-simple then 
$\Gamma$ admits an essentially free measurable 
boundary action, and, 
conversely, $\Gamma$ is $C^*$-simple 
if it admits an essentially free $\mu$-USB
for some $\mu\in\pr(\Gamma)$.
\end{cor}

In~\cite{Bader-Duchesne-Lecureux-16} Bader, Duchesne and Lecureux proved that every amenable IRS 
of a group $\Gamma$ is supported on its 
amenable radical $\Rad(\Gamma)$. 
Consequently, $\Rad(\Gamma)$ is trivial if and only if 
$\delta_{\{e\}}$ 
is the unique invariant probability measure on $\Sub_a(\Gamma)$; 
or that $\Gamma$ has the unique trace property if and only if 
$\Gamma\acts \Sub_a(\Gamma)$ is uniquely ergodic.
Thus, the following is a more concrete evidence that 
the difference between the 
unique trace property and $C^*$-simplicity 
is indeed the difference between 
unique ergodicity and unique stationarity.

\begin{cor}\label{cor:sub_a-unq-stationary}
If $\mu$ is a $C^*$-simple measure on $\Gamma$ then 
$(\Gamma,\mu)\acts (\Sub_a(\Gamma), \delta_{\{e\}})$
is uniquely $\mu$-stationary.

Conversely, if $\Gamma$ is not $C^*$-simple then 
$(\Gamma,\mu)\acts (\Sub_a(\Gamma), \delta_{\{e\}})$
is never uniquely stationary.
\end{cor}

\begin{proof}
Suppose that $\mu$ is $C^*$-simple, and let $\eta$ be an amenable $\mu$-SRS of $\Gamma$. 
As shown in the proof of Theorem~\ref{Zimmer-amenable-action-C*-Simple-group-is-free} 
the function $g\mapsto \eta(\{\Lambda : g\in \Lambda\})$ extends to 
a $\mu$-stationary state $\csr$. 
Thus, by unique stationarity of the canonical trace, we get 
$\eta(\{\Lambda : g\in \Lambda\}) = 0$ for every non-trivial $g\in \Gamma$.
Hence
\[
\eta\left(\Sub(\Gamma)\backslash\{e\}\right) 
= \eta\left(\bigcup_{g\neq e}\, \{\Lambda : g\in \Lambda\}\right) 
\leq \sum_{g\neq e}\, \eta\left(\{\Lambda : g\in \Lambda\}\right) = 0 ,
\]
which implies $\eta = \delta_{\{e\}}$.

Conversely, if $\Gamma$ is not $C^*$-simple, then 
by~\cite[Theorem 1.1]{Kennedy-C*-simplicity-preprint-15} $\Gamma$ has a non-trivial 
\emph{amenable URS} (that is, a minimal subset of $\Sub_a(\Gamma)$ which is not the fixed point $\{e\}$). But any URS supports a 
$\mu$-stationary probability for any $\mu\in\pr(\Gamma)$. 
Hence $\delta_{\{e\}}$ is not unique stationary on $\Sub_a(\Gamma)$ 
for any $\mu\in\pr(\Gamma)$.
\end{proof}

\begin{remark}
Note that since for any $\mu\in\pr(\Gamma)$ 
every URS supports a $\mu$-SRS,  
it follows from the 
above corollary that every amenable URS of 
a $C^*$-simple group $\Gamma$ is trivial.
This is one direction of one of the main results of~\cite{Kennedy-C*-simplicity-preprint-15}. In the proof above, we are using the 
other direction of that result.
\end{remark}

It is natural to ask whether a generalization of the  
result of Bader-Duchesne-Lecureux 
about amenable IRS, which was mentioned above,
holds for stationary random subgroups. 
It is evident that the argument presented in~\cite[Theorem 1.4]{Bader-Duchesne-Lecureux-16}  
cannot be extended to the stationary case, and in fact 
any non-$C^*$-simple group $\Gamma$ with trivial 
amenable radical has a non-trivial amenable SRS.
But rephrasing the above corollary, it implies
$\Gamma$ is $C^*$-simple if and only if 
there is $\mu\in\pr(\Gamma)$ such that 
every amenable $\mu$-SRS of $\Gamma$ is trivial.

\section{Freeness of USB: identifying $C^*$-simple measures}\label{faithfulness-and-freeness-of-USB}

Our proof of the existence of $C^*$-simple measures
(Theorem~\ref{new-characterization-C*-simplicity}) is not
completely constructive.
But, in light of the results of Section~\ref{new-characterization-C*-simplicity-applications}, 
it is natural to ask for concrete examples of $C^*$-simple
measures.
In this section we present several approaches to prove that a given measure is $C^*$-simple.

\subsection{Noetherian actions}\label{Noetherian actions}
For a probability space $(X,\nu)$ we denote by $\malg$ its 
\define{measure algebra}, 
that is, the Boolean algebra of equivalence classes of measurable sets modulo 
$\nu$-null sets. It is a partially ordered set with respect to inclusion, and if $\Gamma$ acts on $(X,\nu)$ then it clearly acts on $\malg$.
\begin{defn}
Let $\Gamma\acts(X,\nu)$ be a measurable action. We say that a collection $\cF \subset \malg$ is a 
\define{$\Gamma$-Noetherian lower semilattice (NLS)} if $\cF$ is $\Gamma$-invariant,
closed under intersections, and any descending chain $Y_1 \geq Y_2 \geq \cdots $ in $\cF$, stabilizes.
\end{defn}

The main result of this section is the following theorem, which is a generaliztion of the well-known fact
that (non-invariant) ergodic stationary actions are atomless, 
and also of~\cite[Lemma 2.2.2]{Kaimanovich-Masur-96}. We are grateful to Uri Bader for suggesting this direction.

\begin{theorem}(0-1 Law for Noetherian actions)\label{zer-one-law}
Let $\Gamma\acts X$, and let $\mu\in \pr (\Gamma)$ be generating. 
Suppose $\nu\in \pr (X)$ is an ergodic $\mu$-stationary measure such that 
$(X,\nu)$ has no non-trivial finite factor.
If $\cF$ is a $\Gamma$-NLS collection in $\malg$, then 
$\nu(Y)\in\{0,1\}$ for any $Y\in\cF$.
\end{theorem}
\begin{proof}
We prove the theorem by Noetherian induction with respect to inclusion. 

Let $Y\in\cF$ and assume $\nu(Z)\in\{0,1\}$ for any $Z < Y$,  
and for sake of contradiction suppose $0<\nu(Y)<1$. 
Then $\nu(Z)=0$ for any $Z < Y$, and in particular,
for any $g\in\Gamma$, either $gY=Y$ or $\nu(Y\cap gY)=0$. 
By non-singularity of $\nu$ it moreover 
follows that for any $g, h\in \Gamma$ either $gY=hY$ or $\nu(gY\cap hY)=0$.
Therefore, by ergodicity, we have 
$1=\nu\left(\bigcup_g gY \right)=\sum_{[g]} \nu(gY)$, 
where $[g]$ is the  equivalence class of all $h$ such that $hY=gY$.
In particular, the set of numbers $\{\nu(gY)\}\subset [0,1]$ has a maximum. 
Observe that these numbers are the values of a bounded harmonic function (the Poisson map of $\mathds{1}_Y$). But a bounded
harmonic function with a maximum must be constant. 
Hence it follows 
there are finitely many $g_1, g_2, \dots, g_n\in\Gamma$ 
such that $X=\bigsqcup_{i=1}^{n}g_iY$,
and $\nu(g_iY)={1}/{n}$ for every $i=1, \dots, n$. 
Since $\nu(Y)<1$ we must have $n>1$. 
But this means $(X,\nu)$ has a 
non-trivial finite factor, which contradicts the assumptions.
\end{proof}

It is well known that $\mu$-boundaries do not admit 
non-trivial invariant factors and hence we get the following.
\begin{cor}\label{triv-free-dichotamy}
Let $(X,\nu)$ be a $\mu$-boundary and assume there is a $\Gamma$-NLS $\cF \subset \malg $ such that $\fix(g)\in\cF$ for all $g\in\Gamma$.
Then every $g\in \Gamma$ acts on $X$ either trivially or essentially freely. 
In particular, if $\Gamma \acts (X,\nu)$ is faithful, it is essentially free.
\end{cor}

\begin{example}{(Algebraic Actions)}\label{example-algebraic}
Let $k$ be a local field, and let $\sG(k)\acts \sV(k)$ be an
algebraic action of a $k$-algebraic group on a $k$-variety.
Let $\Gamma\le\sG(k)$ be a countable subgroup, and let $\cF$ denote
the family of all subvarities of $\sV(k)$. Then $\cF$ is a $\Gamma$-NLS. 
Let $\mu\in\pr(\Gamma)$ be generating, and suppose 
$\nu\in\pr(\sV(k))$ is such that $(\sV(k), \nu)$ is a $\mu$-boundary. Then 
by Theorem~\ref{zer-one-law} every subvariety has trivial $\nu$
measure. 
Furthermore, since the action is algebraic, $\fix(g)\in \cF$ for all 
$g\in\Gamma$ and so by Corollary~\ref{triv-free-dichotamy} 
every faithful $\mu$-boundary algebraic action is essentially free.
\end{example}

Recall that a $\Gamma$-space $X$ is said to be  \define{mean-proximal} if $X$ is a $\mu$-USB for all
generating $\mu\in\pr(\Gamma)$.
We say that a mean-proximal space $X$ is \define{essentially free}
if for every generating $\mu\in\pr(\Gamma)$ 
the action $\Gamma\acts (X, \nu_\mu)$ is essentially free, where 
$\nu_\mu\in\pr(X)$ is the unique $\mu$-stationary measure. 
In the following, to conclude $C^*$-simplicity of the measures in question,
we construct a mean-proximal space 
and verify the essential freeness using the 0-1 Law above.

In the proof of the following theorem, 
we need at certain step to extend an essentially free mean-proximal
action of a finite index normal subgroup. We prove, jointly with Uri Bader, 
the required extension results for 
mean-proximal actions in 
Appendix~\ref{sec:appendix}.

\begin{theorem}\label{thm:linear}
Let $\Gamma$ be a finitely generated linear group with 
trivial amenable radical.
Then any generating
measure on $\Gamma$ is $C^*$-simple.
\end{theorem}

\begin{proof}
Let $\sH$ denote the Zariski closure of $\Gamma$ and denote by $\sH^0 \nor \sH$ the connected component of the identity. 
Since $\Rad(\Gamma)=\{e\}$ we may assume that $\sH^0$ is also semisimple. 
Now let $\Gamma^0=\Gamma\cap\sH^0$. Then $\Gamma^0\nor \Gamma$ and $[\Gamma:\Gamma^0]<\infty$. 
By Theorem~\ref{cor:finite-index} in Appendix~\ref{sec:appendix} below, 
it is enough to find an essentially free metrizable $\Gamma^0$-mean-proximal
space. To construct such a space, we show that for any given $g\ne e\in\Gamma^0$,
there exists a metrizable $\Gamma^0$-mean-proximal space $X_g$ on which $g$ acts essentially freely.
Then the product of all $X_g$ is the desired essentially free
$\Gamma^0$-mean-proximal space, by Lemma~\ref{prod-of-usb-is-usb}.

Fix some $g\ne e\in\Gamma^0$. Observe that $\Rad(\Gamma^0)=\{ e\}$ and hence we can find an element $h$ in the normal
closure of $g$ with an eigenvalue $\alpha$ which is not a root of the identity. Let $R_1$ be the finitely generated ring generated by the entries of the matrices of $\Gamma^0$, and let $R_2$ be the finitely generated ring generated by the polynomials that define $\sH^0$. So $R=\left\langle R_1,R_2,\alpha\right\rangle$ is a finitely generated ring and $I=\{ \alpha^{z}\}_{z\in\bZ}\subset R$ is infinite. 
By a result of Breuillard and Gelander~\cite[Lemma 2.1]{Breuillard-Gelander-07} (see also \cite[Lemma 4.1]{Tits-72}) we can find an embedding $R\subset k$, where $k$ is a local field, such that $I\subset k$ is unbounded. 
Consider $\Gamma^0$ as a subgroup of $GL_{n}(k)$. Then 
$h^z\in\Gamma^0$ has an eigenvalue whose absolute value is greater than $1$ 
for some $z\in\bZ$. 
In particular, it follows $\Gamma^0$ is not relatively compact in $\sH^0(k)$. 
Thus, $\Gamma^0$ is Zariski dense in the $k$-group $\sH^0$, which is connected and semi-simple.
Hence, by results of Margulis \cite[Lemmas IV.4.4 and IV.4.5]{Margulis-book-91}, 
there exist an irreducible representation $\pi:\sH^0(k)\to GL_m(k)$ such that $\pi(h^{z})$ is proximal for some $z\in\bZ$. Thus we may apply another result of Margulis 
\cite[Theorem IV.3.7]{Margulis-book-91} to conclude the projective
space $\bP(k^m)$ is $\Gamma^0$-mean-proximal, and as $\pi(h^z)$ is proximal, it acts non-trivially. Since $h^z$ is in the normal closure of $g$, it follows 
$g$ acts non-trivially on 
$\bP(k^m)$. Now if $\mu\in\pr(\Gamma^0)$ is generating, and  
$\nu\in\pr(\bP(k^m))$ is the unique $\mu$-stationary measure, then by the 0-1 law 
(Theorem~\ref{zer-one-law})
and Example~\ref{example-algebraic}, $g$ acts $\nu$-essentially freely on $\bP(k^m)$.
\end{proof}

\subsection{Groups with countably many amenable subgroups}\label{tits-alternative}

As mentioned before, there are many results 
on concrete realization of the Poisson boundary on a natural USB. 
Taking advantage of this, in the following we  
use some of the main results in this contexts 
to provide further examples of $C^*$-simple measures.

In~\cite{Bowen-Hartman-Tamuz-17} it is shown that if a 
group has only countably many amenable
subgroups, then the action on the (abstract) Poisson 
boundary is essentially free for any generating $\mu$. 
Thus any generating measure $\mu$ on such a group, 
for which there exists a $\mu$-Poisson USB, is a $C^*$-simple measure.

A good source of examples of groups with countably many amenable subgroups is the class of groups satisfying
a ``finitely generated'' version of the Tits alternative. 
By that we mean any subgroup that does not contain a free 
subgroup, is a finitely generated amenable group. 
Examples of such groups are mapping class groups~\cite{McCarthy-85, Ivanov-84, Birman-Lubotzky-McCarthy-83}, and
$\rm{Out}(\bF_n)$~\cite{Bestvina-Feighn-Handel-00, Bestvina-Feighn-Handel-05}.
 
\begin{example}\label{ex:hyperbolic}
Let $\Gamma$ be hyperbolic group and let $\Lambda\le \Gamma$ be a non-elementary subgroup with trivial
amenable radical. Then the Gromov
boundary of $\Gamma$ is an essentially free $\mu$-USB
for any generating measure $\mu$ on $\Lambda$~\cite{Kaimanovich-00}. Hence,
any generating measure on $\Lambda$ is $C^*$-simple.
\end{example}

\begin{thm}\label{mapping-class-groups}
Let $\Gamma$ be a mapping class group.
\begin{enumerate}
\item  If $\Rad(\Gamma)=\{e\}$ then any generating $\mu\in\pr(\Gamma)$ is $C^*$-simple.
\item Let $\Lambda\le \Gamma$ be a non-elementary subgroup with $\Rad(\Lambda)=\{e\}$. Then any generating measure on $\Lambda$ with finite entropy and finite logarithmic moment is $C^*$-simple.
\end{enumerate}
\end{thm}
\begin{proof}
\begin{enumerate}
\item By results of~\cite{Kaimanovich-Masur-96}, the Thurston boundary is mean-proximal 
and the unique stationary measure (for any generating $\mu$)
is supported on minimal projective foliations (in fact, on the 
uniquely ergodic ones).
By~\cite[Theorem 1.4]{Kida-08}, any quasi-invariant probability which is supported on the minimal 
projective foliations is Zimmer-amenable. Since stationary measures are quasi-invariant,
we conclude that the stabilizer of a.e.\ point is amenable. Since there are only 
countably many amenable subgroups, the Thurston boundary is indeed an essentially free mean-proximal space.
\item In~\cite{Kaimanovich-Masur-96} it is proved that
under these assumptions, the Thurston boundary is a Poisson USB of $\Lambda$. In particular, the stabilizers are a.e.\
amenable and hence, similarly as above, the Thurston boundary is an essentially free USB.
\end{enumerate}
\end{proof}

Note that in (1) we do not know that the $\mu$-boundaries
are Poisson boundaries, hence we need other tools to verify
the amenability of the stabilizers.

\begin{theorem}\label{out(Fn)}
Let $\Gamma = \rm{Out}(\bF_n)$ and let $\Lambda\le \Gamma$ be a non-elementary subgroup with $\Rad(\Lambda)=\{e\}$. Then any generating, finitely supported measure on $\Lambda$ is $C^*$-simple.
\end{theorem}
\begin{proof}
By~\cite{Horbez-16} the boundary of the outer space is a Poisson
USB for any finitely supported measure on $\Lambda$. Hence
the result follows similarly to part (2) in Theorem~\ref{mapping-class-groups}.
\end{proof}

In particular, it follows that all non-elementary subgroups of 
a mapping class group or of $\rm{Out}(\bF_n)$ are $C^*$-simple 
(a strengthening of the results of~\cite{Bridson-delaHarpe-04}).

We conclude this discussion by pointing out that if $\mu\in\pr(\Gamma)$ admits 
a Poisson $\mu$-USB, then any $\mu$-boundary which is not the Poisson boundary, 
is not Zimmer-amenable by~\cite[Theorem 9.2]{Nevo-Sageev-13}.
Hence there is no abstract guarantee that the stabilizers would be amenable.
This, in a way, highlights the importance of the tools presented in 
Section~\ref{Noetherian actions} to conclude essential freeness of 
measurable boundary actions. 

\section{Operator-algebraic superrigidity relative to subgroups}\label{Character rigidity}
In this section we study unique stationarity and unique trace property of 
groups $\Gamma$ relative to their subgroups 
$\Lambda\leq\Gamma$, 
and consequently, 
derive several superrigidity results for $\Gamma$ 
relative to $\Lambda$.

\subsection{Unique stationarity relative to subgroups}
Recall that we have canonical inclusions $C^*_\lambda(\Lambda)\subseteq \csr$ and
$C^*(\Lambda)\subseteq C^*(\Gamma)$.
We denote by $\tau_0^\Gamma$ the canonical trace 
on both reduced and full $C^*$-algebras of $\Gamma$.

\begin{definition}
We say a pair $(\Gamma, \Lambda)$ where $\Lambda$ is
a subgroup of $\Gamma$  
is \emph{$\mu$-stationary rigid} 
if the canonical trace $\tau_0^\Gamma$ 
is the unique $\mu$-stationary state on $C^*(\Gamma)$ that 
restricts to the canonical trace $\tau_0^\Lambda$ on $C^*(\Lambda)$. 
\end{definition}

Stationary rigidity of $(\Gamma, \Lambda)$ 
entails strong rigidity properties for $\Gamma$ relative to $\Lambda$.

\begin{thm}\label{rel-stationary-stuck-zimmer}
Suppose $(\Gamma, \Lambda)$ is a $\mu$-stationary rigid pair. 
Then a measurable $\mu$-stationary action $\Gamma\acts(X,\nu)$ is essentially free 
if its restriction to $\Lambda$ is essentially free. 
\end{thm}

\begin{proof}
By Lemma~\ref{random-subgrp-->pos-def}
the function $\phi(g) = \nu(\mathrm{Fix}(g))$ is a pdf on $\Gamma$, 
hence extends to a state $\tau$ on $C^*(\Gamma)$. 
Moreover, as shown in the proof of 
Theorem~\ref{Zimmer-amenable-action-C*-Simple-group-is-free}
the state $\tau$ is $\mu$-stationary.
Since the restriction $\Lambda\acts(X,\nu)$ is essentially free, we have 
$\nu(\mathrm{Fix}(h)) = 0$ for all non-trivial $h\in \Lambda$, which means 
$\phi|_\Lambda=\delta_e$, and equivalently
$\tau|_{C^*(\Lambda)} = \tau_0^\Lambda$.
Hence, $\tau = \tau_0^\Gamma$ by $\mu$-stationary 
rigidity of the pair $(\Gamma, \Lambda)$, which implies $\phi=\delta_e$, 
i.e.\ $\Gamma\acts(X,\nu)$ is essentially free.
\end{proof}

\begin{thm}
Suppose $(\Gamma, \Lambda)$ is a $\mu$-stationary rigid pair. 
Then any non-trivial $\mu$-SRS of $\Gamma$ 
intersect non-trivially with $\Lambda$ with positive probability.
\end{thm}

\begin{proof}
Suppose $\eta$ is a $\mu$-SRS of $\Gamma$.
Then as explained in the proof of Theorem~\ref{rel-stationary-stuck-zimmer}
the pdf $g\mapsto \eta\left(\{\Lambda' \,:\, g\in \Lambda'\}\right)$ 
extends to a $\mu$-stationary state $\tau$ on $C^*(\Gamma)$.
If $\Lambda' \cap \Lambda =\{e\}$ for 
$\eta$-a.e.\ $\Lambda'\in \Sub(\Gamma)$, 
then $\eta\left(\{\Lambda' \,:\, h\in \Lambda'\}\right) = 0$
for every non-trivial $h\in\Lambda$,
hence
$\tau|_{C^*(\Lambda)} = \tau_0^\Lambda$.
Since the pair $(\Gamma, \Lambda)$ 
is $\mu$-stationary rigid it follows 
$\tau = \tau_0^\Gamma$, which implies 
$\eta\left(\{\Lambda' \,:\, g\in \Lambda'\}\right) = 0$ for 
all non-trivial $g\in \Gamma$. 
As shown in the proof of Corollary~\ref{cor:sub_a-unq-stationary} 
this implies $\eta=\delta_{\{e\}}$.
\end{proof}

We may state a von Neumann algebraic relative superrigidity 
for a given stationary rigid pair $(\Gamma, \Lambda)$ 
in the setting of unitary representations of $\Gamma$ whose von Neumann algebras 
admit normal faithful stationary states.
But since at this point we do not have a characterization of 
those von Neumann algebras, 
the significance of such rigidity result in terms of application is not clear, although 
by Example~\ref{type-III} the class of von Neumann algebras 
admitting normal faithful stationary states is strictly larger than the class of finite von Neumann algebras.

But using the fact that stationary states on $C^*$-algebras always 
exist, we obtain 
the following $C^*$-algebraic$\backslash$representation-theoretical relative rigidity.

\begin{prop}
Suppose $(\Gamma, \mathsf{N})$ is a $\mu$-stationary rigid pair, 
where $\mathsf{N}\triangleleft\Gamma$ is a normal subgroup.  
Let $\pi:\Gamma\to\cU(\cH_\pi)$ be a unitary representation.
If $\lambda_\mathsf{N}$ is weakly contained in the restriction 
of $\pi$ to $\mathsf{N}$, then $\lambda_\Gamma$ is weakly contained in $\pi$.
\end{prop}

\begin{proof}
Suppose $\lambda_\mathsf{N}$ is weakly contained in the restriction 
of $\pi$ to $\mathsf{N}$, i.e. there is a canonical 
surjective $*$-homomorphism $C^*_\pi(\mathsf{N})\to C^*_\lambda(\mathsf{N})$. 
In particular, the canonical trace $\tau_0^\mathsf{N}$ is continuous on 
$C^*_\pi(\mathsf{N})$. 
Since $\mathsf{N}$ is normal, $C^*_\pi(\mathsf{N})$ is an invariant 
subalgebra of $C^*_\pi(\Gamma)$ for the inner action by $\Gamma$. 
Thus, by Proposition~\ref{lem:extension-of-statioanry} we can 
extend $\tau_0^\mathsf{N}$ to a $\mu$-stationary 
$\rho\in \cS(C^*_\pi(\Gamma))$. 
Considering $\rho$ as a state on $C^*(\Gamma)$, 
it is still $\mu$-stationary, and thus 
the assumption of $\mu$-stationary rigidity implies $\rho=\tau_0^\Gamma$. 
Hence the canonical trace $\tau_0^\Gamma$ is continuous on 
$C^*_\pi(\Gamma)$, which implies 
there is a canonical 
surjective $*$-homomorphism $C^*_\pi(\Gamma)\to C^*_\lambda(\Gamma)$, 
and equivalently $\lambda_\Gamma$ is weakly contained in $\pi$.
\end{proof}

The following is the main result of this section where we prove  any group that admits an essentially free USB is stationary rigid  
relative to its co-amenable subgroups.
Recall that a subgroup $\Lambda\leq \Gamma$ is co-amenable if 
every affine action of $\Gamma$ with a $\Lambda$-fixed point 
has a fixed point.

\begin{thm}\label{charmonic-rigidity-2}
Let $\mu\in\pr(\Gamma)$. If $\Gamma$ admits an essentially free 
$\mu$-USB, 
then $(\Gamma, \Lambda)$
is a $\mu$-stationary rigid pair for every 
co-amenable subgroup $\Lambda\leq \Gamma$.
\end{thm}

\begin{proof}
Suppose $\tau$ is a $\mu$-stationary state on $C^*(\Gamma)$ 
such that $\tau|_{C^*(\Lambda)} = \tau_0^\Lambda$. 
Let $\pi_\tau:C^*(\Gamma)\to \cB(L^2(C^*(\Gamma), \tau))$ be the GNS representation
of $\tau$. 
Then the Hilbert subspace $L^2(C^*(\Lambda), \tau)$ is canonically isomorphic to 
$L^2(C^*(\Lambda), \tau_0^\Lambda) = \ell^2(\Lambda)$. 
Moreover, $L^2(C^*(\Lambda), \tau)$ is invariant under 
$\pi_\tau(h)$ for every $h\in\Lambda$, and 
$\pi_\tau|_\Lambda: \Lambda \to\cB(L^2(C^*(\Lambda), \tau))$ 
is unitarily equivalent to 
the left regular representation of $\Lambda$. 
Thus, the map $\pi_\tau(h) \mapsto \lambda_h$, $h\in \Lambda$, 
extends to a surjective $*$-homomorphism 
$C^*_{\pi_\tau} (\Lambda) \to C^*_\lambda (\Lambda)$. 
But, since the canonical trace $\tau_0^\Lambda$ coincides with 
$\tau$ on $C^*_{\pi_\tau} (\Lambda)$, and the latter is faithful, 
it follows the canonical surjective $*$-homomorphism above 
is also injective. Hence, by Arveson's Extension Theorem 
\cite[Theorem 1.2.3]{Arveson-69}, 
we may extend the map 
$\lambda_h \mapsto \pi_\tau(h)$, $h\in \Lambda$, to a unital (completely) positive map 
$\cB(\ell^2(\Lambda))\to \cB(L^2(C^*(\Gamma), \tau))$, 
which is then automatically $\Lambda$-equivariant with respect to 
inner actions on both sides. 

Suppose $(X, \nu)$ is an essentially free $\mu$-USB.
Considering $X$ as a $\Lambda$-space via the restriction 
actions, we have a unital positive $\Lambda$-equivariant 
map $C(X) \to \ell^\infty(\Lambda)$. 
Composing with the above map, 
we obtain a unital positive $\Lambda$-equivariant 
map $C(X) \to \cB(L^2(C^*(\Gamma), \tau))$. 

Let $\cE = \{\Phi: C(X) \to \cB(L^2(C^*(\Gamma), \tau))\, |\, \Phi \text{ is positive and unital} \}$,
and define a $\Gamma$-action on $\cE$ by $(g\cdot \Phi)(f) := g(\Phi(g^{-1}f))$.
Then $\cE$ endowed with the point-weak* topology 
(i.e.\ $\Phi_i \to \Phi$ iff $\Phi_i(f)\to \Phi(f)$ weak* for every $f\in C(X)$)
is a compact convex $\Gamma$-space. 

Observe that $\Phi\in \cE$ is a $\Lambda$-fixed point if and only if it is $\Lambda$-equivariant.
So, by the above, $\cE$ contains
a $\Lambda$-fixed point.
Since $\Lambda$ is a co-amenable subgroup, $\cE$ contains also a $\Gamma$-fixed point, 
which is a $\Gamma$-equivariant positive unital map $\iota: C(X)\to \cB(L^2(C^*(\Gamma), \tau))$.

Using Proposition~\ref{lem:extension-of-statioanry} we extend 
$\tau$ to a $\mu$-stationary state
$\tilde\tau$ on $\cB(L^2(C^*(\Gamma), \tau))$.
Then $\tilde\tau\circ\iota$ gives a $\mu$-stationary probability 
on $C(X)$, and 
therefore by the uniqueness assumption we have 
$\tilde\tau|_{\iota(C(X))}=\nu\circ\iota$. 
Hence ${\tilde\tau}_\om\circ{\iota} = ({\tilde\tau\circ\iota})_\om = \delta_{\bnd(\om)}$
for a.e.\ path $\om\in \Omega$, 
where $\bnd: (\Omega,\bP_\mu)\to (X,\nu)$ is 
the boundary map.

Let $g\in\Gamma$ be non-trivial.
Since the action $\Gamma \acts (X,\nu)$ is essentially free, 
$g\bnd(\om) \neq \bnd(\om)$ for a.e.\ path $\om\in \Omega$. 
Consequently, for a.e.\ $\om\in \Omega$ there is 
$f_\om\in C(X)$, $0\leq f_\om\leq 1$, 
with ${\tilde\tau}_\om({\iota}(f_\om))=f_\om(\bnd(\om))=1$ and 
${\tilde\tau}_\om(\pi_\tau(g^{-1}){\iota}(f_\om)\pi_\tau(g)) = f(g\bnd(\om))=0$. 
Hence, Lemma~\ref{lemma0} implies
$\tau_\om(\pi_\tau(g))=\tilde\tau_\om(\pi_\tau(g))=0$
for a.e.\ path $\om\in\Omega$. 
Thus, applying Theorem~\ref{martingales} we get
\[
\tau(\pi_\tau(g)) = \int_\Omega \tau_\om(\pi_\tau(g))\, d\bP_\mu(\om) = 0 .
\]
Hence $\tau = \tau_0^\Gamma$.
\end{proof}

Applying results of Section~\ref{tits-alternative} we get the following. 

\begin{cor}
Let $\Gamma$ be a hyperbolic group, a mapping class group, or 
a finitely generated linear group, and assume that 
$\Rad(\Gamma) =\{e \}$. Then
for any co-amenable subgroup $\Lambda\leq \Gamma$ the pair $(\Gamma,\Lambda)$ is 
$\mu$-stationary rigid for any generating $\mu$.
\end{cor}

\subsection{Unique trace property relative to subgroups}
In this section we consider a relative unique ergodicity for 
the canonical trace.
This should be considered as a relative character rigidity property.

\begin{definition}
We say a pair $(\Gamma, \Lambda)$ of a group $\Gamma$ and a subgroup $\Lambda$ 
is \emph{tracial rigid} 
if the canonical trace $\tau_0^\Gamma$ on $C^*(\Gamma)$ 
is the unique tracial extension of 
the canonical trace $\tau_0^\Lambda$ on $C^*(\Lambda)$. 
\end{definition}

Recall that a von Neumann algebra is \emph{finite} if it has a normal faithful trace.

\begin{thm}\label{relative-op-alg-superrigid}
Suppose $(\Gamma, \Lambda)$ is a tracial rigid pair, and $\Lambda$ is an icc group.
Suppose $\pi:\Gamma\to\cU(\cH_\pi)$ is a unitary representation such that 
$M=\pi(\Gamma)''$ is a finite von Neumann algebra. 
If the restriction $\pi|_\Lambda$ extends to 
an isomorphism 
$\cL(\Lambda) \cong \pi(\Lambda)''$ then $\pi$ extends to 
an isomorphism $\cL(\Gamma) \cong M$.
\end{thm}

\begin{proof}
Since $\Lambda$ is icc, 
$\pi(\Lambda)'' \cong \cL(\Lambda)$ is a factor, 
and in particular admits a unique trace,
which is the canonical trace $\tau_0^\Lambda$.
Now suppose $\tau$ is a normal trace on $M$, then 
its restriction to $C^*_\pi(\Gamma)$ is the canonical trace by 
tracial rigidity of the pair $(\Gamma, \Lambda)$.
Since $C^*_\pi(\Gamma)$ is weak* dense in $M$, it follows 
$\tau=\tau_0^\Gamma$ is the unique trace on $M$, thus also faithful. 

Now let $\iota_\tau: M\to \cB(L^2(M, \tau))$ denote the GNS map.
Then the map $\delta_g \to \iota_\tau(\pi(g))$ extends to a unitary ${U}_\tau$ from
$\ell^2(\Gamma)$ onto $L^2(M, \tau)$, and we have
\[
{U}_\tau^* \pi(g) {U}_\tau \delta_h = 
{U}_\tau^* \pi(g) \iota_\tau(\pi(h)) =
{U}_\tau^*  \iota_\tau(\pi(gh)) =
\delta_{gh} =
\lambda_{g} \delta_{h}
\]
for all $g, h\in \Gamma$. Hence 
$\operatorname{Ad}({U}_\tau) : {\mathcal{L}}({\Gamma}) \to M$ 
is the desired isomorphism.
\end{proof}

\begin{thm}\label{pmp-free-if-rest-free}
Suppose $(\Gamma, \Lambda)$ is a tracial rigid pair. 
Then a probability measure preserving action $\Gamma\acts(X,m)$ 
is essentially free if its restriction $\Lambda\acts(X,m)$ is essentially free.

Equivalently, any non-trivial IRS of $\Gamma$ intersects $\Lambda$ non-trivially, with positive probability.
In particular, every non-trivial normal subgroup $\mathsf{N}\triangleleft \Gamma$
intersects $\Lambda$ non-trivially. 
\end{thm}

\begin{proof}
For $g\in\Gamma$ denote $\mathrm{Fix}(g) = \{x\in X : gx=x\}$.
Then the function $g\mapsto m(\mathrm{Fix}(g))$ is 
positive definite on $\Gamma$ by Lemma~\ref{random-subgrp-->pos-def}, 
hence extends to a trace $\tau$ on $C^*(\Gamma)$.
Moreover, by invariance of $m$ we get
\[
m(\mathrm{Fix}(hgh^{-1})) = m(h\mathrm{Fix}(g)) = 
m(\mathrm{Fix}(g))
\]
for all $g, h\in \Gamma$,
which implies $\tau$ is a trace. 
If the restriction $\Lambda\acts(X,m)$ is essentially free, then 
$m(\mathrm{Fix}(h)) = 0$ for every non-trivial $h\in \Lambda$,
and equivalently $\tau|_{C^*(\Lambda)} = \tau_0^\Lambda$. 
Thus by the assumption of tracial rigidity of 
the pair $(\Gamma, \Lambda)$ we get 
$\tau= \tau_0^\Gamma$, hence the action $\Gamma\acts(X,m)$ 
is essentially free.

For the IRS formulation, it is well known that any given IRS
is the push-forward of some measure preserving action
via the stabilizer map $x \mapsto \stab_{\Gamma}(x)$ 
(see for example~\cite{Abert-Glasner-Virag-14}).
\end{proof}

By working with topological boundaries instead 
of unique stationary measurable boundaries we are able 
to generalize Theorem~\ref{charmonic-rigidity-2} 
for the case of tracial pairs.

\begin{thm}\label{thm:co-amenalbe-are-character-rigid}
Suppose $\Lambda\leq \Gamma$ is co-amenable. 
Then every tracial extension of 
the canonical trace $\tau_0^\Lambda$ to $C^*(\Gamma)$ 
is supported on $\Rad(\Gamma)$. 

In particular, if $\Gamma$ has trivial amenable radical,
then the pair $(\Gamma, \Lambda)$ is tracial rigid. 
\end{thm}

\begin{proof}
Suppose $\tau$ is a trace on $C^*(\Gamma)$ 
such that $\tau|_{C^*(\Lambda)} = \tau_0^\Lambda$. 
Denote by $\pi_\tau:C^*(\Gamma)\to \cB(L^2(C^*(\Gamma), \tau))$ 
the GNS representation of $\tau$. 

Let $g\notin\Rad(\Gamma)$, then it follows from \cite{Furman-03} 
that $\Gamma$ admits a topological boundary action 
$\Gamma\acts X$ on which $g$ acts non-trivially. 
Now, an argument similar to the 
proof of Theorem~\ref{charmonic-rigidity-2} 
yields a $\Gamma$-equivariant unital positive map 
$\iota:C(X)\to \cB(L^2(C^*(\Gamma), \tau))$. 
Extend the trace $\tau\in \cS(\pi_\tau(C^*(\Gamma)))$ 
to a state $\rho$ on $\cB(L^2(C^*(\Gamma), \tau))$. 
Let $x\in X$ be such that $gx\neq x$, 
and choose $f\in C(X)$, $0\leq f\leq 1$, with $f(x)=1$ and $f(gx)=0$. 
Consider the restriction $\rho|_{\iota(C(X))}$ as a probability on $X$.
Since $\Gamma\acts X$ is a boundary action, 
there is a net $(g_i)\subset \Gamma$
such that $g_i \rho|_{\iota(C(X))} \to \delta_x$. 
By passing to a subnet, if necessary,
we may assume $g_i \rho$ converges weak* 
to a state $\eta$ on $\cB(L^2(C^*(\Gamma), \tau))$.
Then $\eta|_{\iota(C(X))} = \delta_x$, and therefore 
$\eta(\iota(f)) = f(x) = 1$ and 
$\eta(\pi(g^{-1})\iota(f)\pi(g)) = \eta(\iota(g^{-1} f)) = f(gx) = 0$.
Hence Lemma~\ref{lemma0} yields
$\tau(\pi(g))=\eta(\pi(g))=0$. This shows $\tau = \tau_0^\Gamma$.
\end{proof}

\begin{cor}
Let $\Gamma$ be a group with trivial amenable radical. 
Suppose $\Lambda\leq \Gamma$ is co-amenable, and 
suppose $\Lambda$ is a character rigid group.
Then a probability measure preserving action of $\Gamma$ 
on the standard 
Lebesgue space is essentially free if its restriction 
to $\Lambda$ is ergodic.
\end{cor}

\begin{proof}
Any ergodic measure preserving action of 
a character rigid group is essentially free. 
Thus, the assertion follows immediately from 
Theorems~\ref{pmp-free-if-rest-free} and~\ref{thm:co-amenalbe-are-character-rigid}.
\end{proof}

\newpage

\appendix 
\section{Extending mean-proximal actions}\label{sec:appendix}
\newcommand{\aut}{\operatorname{Aut}(\Gamma)}
\begin{center}
Uri Bader, Yair Hartman and Mehrdad Kalantar
\end{center}
{\ }

In this appendix we prove an extension theorem for mean-proximal actions 
that is needed in our proof of $C^*$-simplicity of generating measures 
on linear groups (Theorem~\ref{thm:linear}).
Our results below can be proven in more general forms, but we state them 
as we need them in this work.

Throughout this section, $\Gamma$ is a countable discrete group. 
By a $\Gamma$-space we mean a compact space $X$ on which 
$\Gamma$ acts by homeomorphims. A $\Gamma$-space
$X$ is said to be \define{$\Gamma$-mean-proximal}, 
if for every generating $\mu \in \pr(\Gamma)$, there exists a unique $\mu$-stationary measure
$\nu_{\mu}\in \pr (X)$ such that $(X,\nu_{\mu})$ is a $(\Gamma,\mu)$-boundary 
(see Definition~\ref{defn:measurable-boundaries}).

\begin{lemma}\label{prod-of-usb-is-usb}
Let $\{X_j\}$ be a countable collection of metrizable $\Gamma$-mean-proximal
spaces.
Then $\prod_j X_j$ equipped with the diagonal action is a $\Gamma$-mean-proximal space. 
\end{lemma}

\begin{proof}
Fix a generating $\mu\in\pr(\Gamma)$, and let $(\Omega, \bP_{\mu})$ 
denote the path space of the $(\Gamma, \mu)$-random walk
(see Section~\ref{Random walks and stationary dynamical systems}). 
For each $j'$, let $\pi_{j'}:\prod_j X_j\to X_{j'}$
be the corresponding projection map.
Let $\nu$ be a $\mu$-stationary probability on $\prod_j X_j$, and 
for each $j$ denote by  
$\nu^j$ the pushforward of $\nu$ under $\pi_j$. 
Then $\nu^j\in \pr(X_j)$ is $\mu$-stationary, hence by
the uniqueness assumption, is also a $(\Gamma,\mu)$-boundary.
Since the collection $\{X_j\}$ is countable 
we can find $\Omega_0 \subset \Omega$ 
with $\bP_{\mu}(\Omega_0)=1$ such that for every $\omega\in\Omega_0$, 
the conditional measures $(\nu^j)_{\omega}$ 
are Dirac measures for all $j$ (see Furstenberg's Theorem~\ref{cond-meas} for the definition). 
Since the projections $\pi_j$ are equivariant, we observe 
$(\nu_{\omega})^j=(\nu^j)_{\omega}$ for all $j$ and $\om\in \Omega_0$. 
Thus, it follows that $\nu_{\omega}$
is a Dirac measure for all $\om\in \Omega_0$ (note that the product space is also metrizable).
This implies $(\prod_j X_j, \nu)$ is a $(\Gamma,\mu)$-boundary. 
The uniqueness of $\nu$ follows from the fact that every stationary measure 
is completely determined by its conditional measures 
(Theorem~\ref{cond-meas}), and that the 
conditional measures $\nu_{\omega}$ project onto the Dirac measures $(\nu^j)_{\omega}$.
\end{proof}

\begin{thm}\label{prop:extension}
Let $\Gamma$ be a finitely generated group, and let $X$ be 
a metrizable $\Gamma$-mean-proximal space.
Then there exists a metrizable $\Gamma$-mean-proximal space $Y$
such that $X$ is a factor of $Y$, and the action $\Gamma\acts Y$ extends 
to an action $\aut\acts Y$.
\end{thm}

\begin{proof}
For $\alpha\in\aut$ let $X_{\alpha}$ be a copy of $X$ equipped
with the action $g\cdot x=\alpha^{-1}(g)x$. Obviously $X_{\alpha}$ is 
a metrizable $\Gamma$-mean-proximal for every $\alpha\in\aut$. 
Since $\Gamma$ is finitely generated, 
$\aut$ is countable, hence Lemma~\ref{prod-of-usb-is-usb} yields that 
$Z= \prod_{\alpha} X_{\alpha}$ 
equipped with the diagonal action of $\Gamma$ is 
mean-proximal. In particular, $Z$ contains a unique minimal component $Y$, 
which is also $\Gamma$-mean-proximal, and hence 
a topological boundary by Theorem~\ref{USB->top-bnd}.

On the other hand, the group $\aut$ also acts naturally on $Z$ by permuting the indices, 
namely $(\beta \cdot z)_{\alpha_0} = z_{\beta^{-1}\alpha_0}$, for $\beta, \alpha_0 \in \aut$ and 
$z= (z_\alpha)_{\alpha\in \aut}\in Z$. One observes the relation 
$g\cdot\beta\cdot z = \beta \cdot\beta^{-1}(g)\cdot z$ for $g\in \Gamma$, $\beta\in\aut$, 
and $z\in Z$. Thus, it follows if $Z'\subset Z$ is $\Gamma$-invariant, then 
so is $\beta(Z')$ for all $\beta\in\aut$. In particular, 
$\beta(Y) = Y$ for every $\beta\in \aut$ by the uniqueness. Hence the restriction 
to $Y$ defines an action $\aut\acts Y$.
 
Now fix $g\in\Gamma$, and let $\beta_g\in \aut$ 
be the inner automorphism $\beta_g(h)=ghg^{-1}$. 
Considering $g$ and $\beta_g$ as homeomorphism on $Z$ via the 
above actions of $\Gamma$ and $\aut$, straightforward calculations show
the composition $g^{-1}\beta_g:Z\to Z$ is equivariant with respect to the 
diagonal action of $\Gamma$. 

In particular, we have $g^{-1}\beta_g(Y) = Y$ by
the uniqueness. 
But since $\Gamma\acts Y$ is a topological boundary action, it follows $g^{-1}\beta_g$ 
is the identity map on $Y$, which implies the restriction of the action $\aut\acts Y$ to 
$\Gamma\leq \aut$ (via inner automorphism) 
coincided with the diagonal action of $\Gamma$ when restricted to $Y$, 
and this completes the proof.
\end{proof}

We will use the notion of recurrent subgroups in order to 
relate boundary actions of a group to its finite index subgroups.
We recall the definition, and refer the reader to~\cite{Furstenberg-73} 
for more details.

Let $\mu\in\pr(\Gamma)$ be generating. A subgroup $\Lambda\le \Gamma$ is said to be \define{$\mu$-recurrent} 
if almost every path of the $(\Gamma, \mu)$-random walk 
passes through $\Lambda$ (infinitely many times, automatically).
In this case, there exists a \emph{hitting measure} $\theta\in\pr(\Lambda)$
such that the restriction map $\ell^\infty(\Gamma)\to \ell^{\infty}(\Lambda)$
yields an isometric isomorphism between the algebras of bounded 
harmonic functions $H^{\infty}(\Gamma,\mu)$ and $H^{\infty}(\Lambda,\theta)$.

\begin{lem}\label{lem:recurrent}
Let $X$ be a metrizable $\Gamma$-space and let 
$\Lambda\le \Gamma$ be $\mu$-recurrent for a generating $\mu\in\pr(\Gamma)$.
Then any $\mu$-stationary measure $\nu\in\pr(X)$ is $\theta$-stationary, 
where $\theta\in \pr(\Lambda)$ is the hitting measure.
Moreover, if $\nu\in\pr(X)$ is a $(\Lambda,\theta)$-boundary then it is also 
a $(\Gamma,\mu)$-boundary.

In particular, if $(X,\nu)$ is a $\theta$-USB, then it is also a $\mu$-USB.
\end{lem}

\begin{proof}
For a proof of the first assertion see~\cite[Corollary 2.14]{Hartman-Lima-Tamuz-14}. 
The second part follows directly from the definitions.
\end{proof}

We say a $\Gamma$-mean-proximal space $X$ is \define{essentially free}
if for any generating measure $\mu\in\pr (\Gamma)$, the action $\Gamma \acts (X,\nu_{\mu})$
is essentially free.

\begin{thm}\label{cor:finite-index}
Let $\Gamma$ be a finitely generated group with trivial amenable radical, and 
let $\Lambda$ be a normal subgroup of $\Gamma$ of finite index. 
If $\Lambda$ admits an essentially free metrizable
mean-proximal action, then so does $\Gamma$.
\end{thm}
\begin{proof}
Let $X$ be an essentially free metrizable $\Lambda$-mean-proximal space, and 
let $Y$ be the $\Gamma$-space given by Theorem~\ref{prop:extension} 
(considering $\Gamma\leq \mathrm{Aut}(\Lambda)$). 
Since $Y$ is a $\Lambda$-extension of $X$, it is $\Lambda$-essentially free. 

Since $\Lambda$ has finite index in $\Gamma$, it is $\mu$-recurrent for any 
generating $\mu\in \pr(\Gamma)$, thus
Lemma~\ref{lem:recurrent} implies that $Y$ is also $\Gamma$-mean-proximal.
To see that $Y$ is also $\Gamma$-essentially free, let $\mu\in\pr(\Gamma)$ be generating, 
and let $\theta\in\pr(\Lambda)$ be 
the corresponding hitting measure. Denote by $\nu=\nu_{\mu}=\nu_{\theta}\in\pr(Y)$ 
the unique stationary measure.

Since $\Lambda \acts (X,\nu)$ is essentially free, 
the stabilizers $\stab_{\Gamma}(x)$  are finite subgroups of $\Gamma$
for $\nu$-almost every $x$. 
Thus, the pushforward $\eta$ of $\nu$ under the stabilizer map 
$x \mapsto \stab_\Gamma(x)$ is 
a stationary measure on the compact space $\Sub(\Gamma)$ (of all subgroups 
of $\Gamma$, 
endowed with the conjugate action of $\Gamma$) that is supported on finite subgroups.
Note that $\Gamma$ has only countably many finite subgroups, 
hence $\eta$ is invariant (see, e.g.\ the discussion in Section~\ref{tits-alternative}), i.e.\ an IRS. 
But by~\cite{Bader-Duchesne-Lecureux-16} every amenable IRS is supported 
on the amenable radical of $\Gamma$, which is trivial by assumption. 
Hence $\stab_{\Gamma}(x)$ is trivial for $\nu$-almost every $x$, and this finishes the proof. 
\end{proof}

\begin{remark}
The assumption of metrizability was not necessary in 
Lemmas~\ref{prod-of-usb-is-usb} and \ref{lem:recurrent}, 
and was only needed in Theorem~\ref{prop:extension} to ensure metrizability of 
the space $Y$. 
In all conclusions, the general case can be reduced to the 
metrizable case by either passing to a metrizable model, or 
considering metrizable factors and taking an inverse limit. 
Since we only use the results of this appendix in situations 
where all spaces in considerations are metrizable, we did not 
take those extra steps and just tailored the statements for our particular 
purposes here. 
But, the fact that we can conclude these results for arbitrary products 
and inverse limits leads to a deeper point: there is a theory of mean-proximal 
actions parallel to those of proximal and strongly proximal actions. 
One can prove the existence of a universal action, and canonical extension results. 
We intend to devote a followup paper to a conceptual study of 
a class of topological dynamical systems that include all these examples.

\end{remark}

\end{document}